\documentclass[11pt]{article}
\usepackage{hyperref}
\usepackage{amsfonts}
\usepackage{amssymb,amsmath}
\usepackage{ stmaryrd }
\usepackage[utf8]{inputenc}
\usepackage[english]{babel}
\usepackage{amsthm}
\usepackage[a4paper, total={6in, 8in}]{geometry}
\usepackage{color}

\usepackage{graphicx}
\def\be{\begin{eqnarray}}
\def\ee{\end{eqnarray}}
\def\ben{\begin{eqnarray*}}
\def\een{\end{eqnarray*}}

\def\me{\medskip\noindent}

\def\ben#1{\color{green}#1 \color{black}}

\newcommand{\Co}{\mathcal{C}}

\def\D{\mathbb{D}}

\def\N{\mathbb{N}}
\def\P{\mathbb{P}}
\def\R{\mathbb{R}}

\def\E{\mathbb{E}}

\def\s{\sigma}

\def\ind{{\mathchoice {\rm 1\mskip-4mu l} {\rm 1\mskip-4mu l}
{\rm 1\mskip-4.5mu l} {\rm 1\mskip-5mu l}}}

\newtheorem{thm}{Theorem}[section]
\newtheorem{lem}[thm]{Lemma}
\newtheorem{cor}[thm]{Corollary}

\newtheorem{prop}[thm]{Proposition}

\renewenvironment{proof}{\noindent {\bf Proof \phantom{9}}}
{\hfill $\square$ \vspace{0.25cm}}

\title{Dynamics of lineages in adaptation to a gradual environmental change}

\author{Vincent Calvez\thanks{ICJ, UMR 5208 CNRS \& Universit\'e Claude Bernard Lyon 1, Lyon France; E-mail: \texttt{vincent.calvez@math.cnrs.fr}},  Beno\^{i}t Henry\thanks{IMT Lille Douai, Institut Mines-T\'el\'ecom, Univ.\ Lille, F-59000 Lille, France; E-mail: \texttt{benoit.henry@imt-lille-douai.fr}}, Sylvie
    M\'el\'eard\thanks{Ecole Polytechnique, CNRS, UMR 7641 -  CMAP, route de
    Saclay, 91128 Palaiseau Cedex-France; E-mail: \texttt{sylvie.meleard@polytechnique.edu}}, Viet Chi Tran\thanks{LAMA, Univ Gustave Eiffel, Univ Paris Est Creteil, CNRS, F-77454 Marne-la-Vall\'ee, France; E-mail:
    \texttt{chi.tran@univ-eiffel.fr}}}

\date{\today}

\begin{document}

\maketitle

\begin{abstract}
We investigate a simple quantitative genetics model subjet to a gradual environmental change from the viewpoint of the phylogenies of the living individuals. We aim to understand better how the past traits of their ancestors are shaped by the adaptation to the varying environment. The individuals are characterized by a one-dimensional trait. The dynamics -births and deaths- depend on a time-changing mortality rate that shifts the optimal trait to the right at constant speed. The population size is regulated by a nonlinear non-local logistic competition term. The macroscopic behaviour can be described by a PDE that admits a unique positive stationary solution. In the stationary regime, the population can persist, but  with a  lag in the trait distribution due  to the environmental change. For the microscopic (individual-based) stochastic process, the evolution of the lineages can be traced back using the historical process, that is, a measure-valued process on the set of continuous real functions of time. Assuming stationarity of the trait distribution, we describe the limiting distribution, in large populations, of the path of an individual drawn at random at a given time $T$. Freezing the non-linearity due to competition allows the use of a many-to-one identity together with
Feynman-Kac's formula. This path, in reversed time, remains close to a simple Ornstein-Uhlenbeck process. It shows how the lagged bulk of the present population stems from ancestors once optimal in trait but still in the tail of the trait distribution in which they lived.
\end{abstract}

\me Keywords: stochastic individual-based model, spine of birth-death process, many-to-one formula, ancestral path, historical process, genealogy, phylogeny, resilience.

\bigskip

\noindent \textbf{Acknowledgements:} The authors thank Pierre-Louis Lions for his proof of the uniqueness of the stationary distribution. This work has been supported by the Chair ``Mod\'elisation Math\'ematique et Biodiversit\'e'' of Veolia Environnement-Ecole Polytechnique-Museum National d'Histoire Naturelle-Fondation X. V.C.T. also acknowledges support from Labex B\'ezout (ANR-10-LABX-58).
This project has received funding from the European Research Council (ERC) under the European Union’s Horizon 2020 research and innovation programme (grant agreement No 865711).

\section{Introduction}

An increasing number of studies have demonstrated rapid phenotypic changes in invasive or natural populations that are subject to environmental changes, such as climate change or habitat
alteration due for instance to human activities \cite{BraHol06,Par06,Hen08,HofSgr11,burrows_pace_2011,anderson_evolutionary_2012,sheldon_climate_2019}. Such fast evolution may result in adaptation for these populations facing changing environment, as observed in evolutionary experiment~\cite{Gon13,GorAarZwa16,Col18,guzella_slower_2018}.
Important theoretical progress has been made to predict phenotypic evolution in a changing environment since the pioneer works of~\cite{LynGabWoo91,LynLan93,BurLyn95,LanSha96}, see \cite{KopMat13} for a review.\\
A major prediction of these models is that when the optimal phenotype changes linearly with time, the phenotypical distribution of individuals is moving at the same speed to keep pace of the change, but is lagged behind the optimum. The equilibrium value of the lag depends on the rate of the change, on the genetic variance and on the strength of selection. Above a critical rate of change of the optimal phenotype with time, this lag becomes so large that the fitness of the population falls below the value that allows its persistence and the population is doomed to extinction. In the case of small enough lag such that the population persists under constant adaptation, we study, in the present paper, the genealogies of its individuals. Our aim is to understand how individual dynamics build the macroscopic adaptation of the population \textit{via} the quantitative description of the typical ancestral lineage.\\

We consider a population dynamics where individuals are characterized by a trait $x\in \mathbb{R}$ and that give birth and die in continuous time. During their life the individual trait variations are modelled  by a diffusion operator with variance $\sigma^2$. The environment   in which the population lives is shifted at constant speed $\sigma c>0$ that drives the adaptation of the population. The constant $c>0$ can be interpreted as the speed of environmental change.\\
{Two complementary descriptions of the dynamics are opted for: (i) a macroscopic, deterministic, description of the phenotypic density, and (ii) a microscopic, stochastic, description of the individual phenotypes. We believe that the later is better suited for the analysis of the lineages.}\\

The macroscopic dynamics of the phenotypic density in the moving environment is described by  the following partial differential equation (PDE) for the density with respect to the one-dimensional trait:
\begin{equation}\label{pde0}
\partial_{t}u(t,x)=\frac{\sigma^2}{2}\partial^2_{xx} u(t,x)+\left(1-\frac{1}{2}(x-\sigma ct)^2-\int_{\mathbb{R}}u(t,y)dy  \right)u(t,x).
\end{equation}
Here, $u(t,x)$ denotes the density of population at time $t$ and trait $x\in \R$. Due to the environmental change, the optimal trait with regards to the growth rate $1-(x-\sigma ct)^2/2$ is $x^*_t=\sigma c t$. The choice of the scaling of the speed is discussed after Theorem \ref{reversed} below. The nonlinear term involving the total mass of the population accounts for the mean field competition between individuals at time $t$.

It is  well known (cf. \cite{champagnat_microscopic_2006,champagnatferrieremeleard_stochmodels,fourniermeleard}) that  equation \eqref{pde0} can be derived from a stochastic system describing the random individual dynamics. More precisely, we consider the following  branching-diffusion process with interaction. An individual, at trait $x\in \R$ at time $t\geq 0$, gives birth to a new individual at the same  trait with rate 1.  Each individual dies with rate $(x-\sigma ct)^2/2+N_t/K$,  where $N_t$ is the total population size at time $t$ and $K$ is the carrying capacity of the system. The natural death rate $(x-\sigma ct)^2$ reflects the gradual environmental change, as in the PDE. The term $N_t/K$ in the death rate corresponds to density-dependent competition. Changes in the trait during the lives of individuals are driven by independent Brownian motions, accounting for infinitesimal changes of the phenotypes. It is standard to rigorously prove that the empirical measure on the individual traits  weighted by $1/K$ satisfies a semi-martingale decomposition, which is a stochastic equation analogous to \eqref{pde0} (and given later), and that it converges weakly to the solution of the PDE when $K$ tends to infinity (provided the initial conditions are scaled suitably).\\

\smallskip Due to the environmental change, the behavior of the population is naturally observed  in the moving frame. In what follows, we will always work in this setting, defining the density in the moving frame as
 $f(t,x)=u(t,x+\sigma c t)$. As such, we obtain an additional transport term in the PDE (associated with a drifted Brownian motion in the individual-based model):

 \begin{equation}\label{pde}
\partial_{t}f(t,x)=\frac{\sigma^2}{2}\partial^2_{xx} f(t,x)+ \sigma c \ \partial_{x}f(t,x) +  \left(1-\frac{x^2}{2}-\int_{\mathbb{R}}f(t,y)dy  \right)f(t,x).
\end{equation}

The unique positive stationary state of this equation can easily be computed. It can exist if, and only if, $1-\sigma/2-c^2/2>0$, which is the persistence condition on the speed $c$. Under this condition, the stationary state is a  weighted Gaussian density centered on $-c$, with variance $\sigma$, hereafter denoted by $F$ (see Section \ref{sec:F}). The shift by $c$ relative to the fitness optimum at $x=0$ can be interpreted as a  lag in the process of adaptating to a moving environment. Indeed individuals try keeping pace of the gradual change, so that they can never be optimal in average. This maladaptation can be measured by the shift $c$ which is associated with a load in the fitness of value $c^2/2$.

Additionally, this model predicts that the population collapses when the speed of environmental change is above a certain threshold $c^* =  (2 - \sigma)^{1/2}$. Here, we consider that the speed $c$ is below $c^*$, as already mentioned above.

Our purpose here is to provide more insight on this phenomenon by studying the trait ancestry of the individuals at a given time $T$, \textit{i.e.} the sequence of traits of their ancestors in the past.
We assume that, in the moving frame, the population dynamics are nearly stationary, starting from the equilibrium $F$.  In particular, the solution $f(t,x)$ of the deterministic PDE \eqref{pde} remains  constant in time, equal to this equilibrium. Consequently, the stochastic process will stay close to this equilibrium on finite time intervals in the regime of large population.
In the stationary regime the dynamics of the PDE is trivial but the dynamics of the lineages are not, as can be seen on numerical simulations of the individual-based models (cf.\ Fig. \ref{fig:1}, both in the original variables, and in the moving frame).

\begin{figure}[!ht]
\begin{center}
\begin{tabular}{cc}
\includegraphics[width=.45\linewidth]{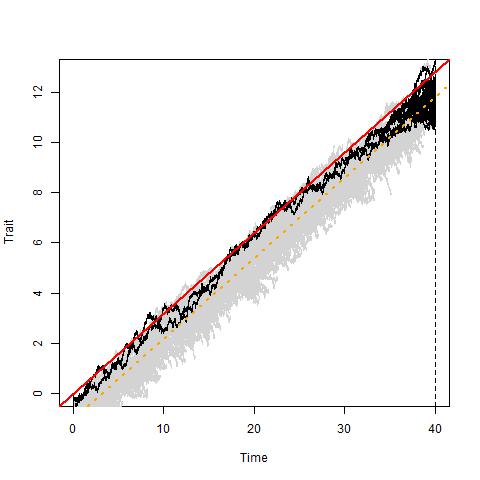} &
\includegraphics[width=.45\linewidth]{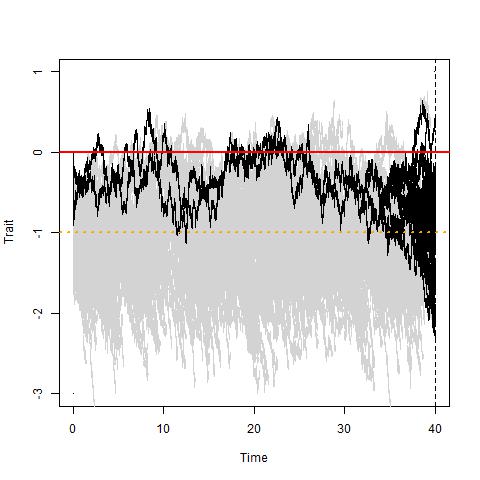} \\
(a) & (b)
\end{tabular}
\caption{{\small Ancestral lineages of the present population. To an individual of trait $x$ living at time $T$, its lineage corresponds to the function that associates with each time $t<T$ the trait of this individual if it was already born, or else the trait of its closest ancestor at that time (its parent if the latter was born, otherwise its grand-parent etc.). The traits in the population (ordinate) are shown with respect to time (abscissa). The extinct lineages are in gray, whereas the lineages of the living particles are in black. As can be seen, the trait distribution is nearly stationary (gray background on the right image), whereas the lineages follow an Ornstein-Uhlenbeck process (see our main result, Theorem \ref{reversed}). (a): fixed frame. (b): moving frame. The parameters are $c=1$, $\sigma=0.32$, $K=250$.}}\label{fig:1}
\end{center}
\end{figure}

 We  observe the following pattern: a stabilized cloud of points representing the stationary state and solid lines representing the lineages, highlighting the response to environment. One observes that the individuals alive at the final observation time are all coming from past individuals whose traits were far from being representative in the past distribution but who were better fitted.

 \smallskip More precisely, we will describe the approximate dynamical lineage of a fixed individual sampled uniformly in a large population at a time $T>0$. We will show the following Theorem \ref{reversed} stating that in backward time, these trajectories are asymptotically (when {the carrying capacity} $K$ tends to infinity), Ornstein-Uhlenbeck processes.
\begin{thm}\label{reversed}
In the moving framework, assuming the trait distribution of the population stationary, the backward in time process describing the lineage of an individual sampled in the living population at time $T>0$ converges, when $K\rightarrow +\infty$, to the following time homogeneous Ornstein-Uhlenbeck process driving the ancestral trajectories around $0$, according to the equation
\begin{equation}\label{OU-intro}
d\widehat{Y}_{s}=-\sigma \widehat{Y}_{s}ds +\sigma d W_{s},
\end{equation}
for a Brownian motion $W$.
\end{thm}
This result is made more precise in Theorem \ref{thm:conclusion}. A similar conclusion was derived for a similar model, independently of this work, by another approach in \cite{ForienGarnierPatout}. The latter analysis remains on a macroscopic level and follows the tracking of neutral fractions in the PDE, as initiated in \cite{Roquesetal}. \\

It is an immediate observation that the Ornstein-Uhlenbeck process is independent of the speed $c$. This is indeed due to our choice of scaling the speed of change $c$ by the standard deviation $\sigma$ in \eqref{pde0} and \eqref{pde}. This is to say that the speed of change is measured relatively to how many units of standard mutational deviation are shifted per time unit. With this scaling, the lag load $c^2/2$ is independent of the mutational variance rate $\sigma^2$. In particular it does not vanish as the mutational variance goes to zero. \\

Although we cannot handle the long time asymptotics with our methodology,  we can still notice that the stationary distribution of the backward Ornstein-Uhlenbeck process is another Gaussian distribution centered at the origin, with variance $\sigma/2$. Hence, individuals sampled at time $T$ come from ancestors that were close to being optimal in the past, but not representative in the distribution at that time (see also  Fig. \ref{fig:1}).\\
Notice that in the extreme case of a vanishing variance $\sigma^2\to 0$, a simple long time scaling $s' = \sigma s$ in the SDE \eqref{OU-intro} makes it close to the deterministic ODE $d\widehat{Y}_{s'}=- \widehat{Y}_{s'}ds'$. The solution of the latter equation converges in long time to 0.\\
{This study shows how important it is in ecology or agriculture to preserve the trait diversity, as the subpopulation with the majority trait may not be the one ensuring the survival of the species in case of environmental shift. In cancer therapy or for understanding antibiotic resistances, our results show that the eradication of such majority trait with gradual effects of drugs or antibiotics may not be enough to fight against the persistence of tumors or bacterial strains. We also refer to \cite{guzella_slower_2018} for similar consideration in experimental evolution.\\

 The proof of Theorem \ref{reversed} is now sketched. Our approach mixes here two points of view based on the stochastic individual-based model: on the one hand, the spinal approach as developed for branching diffusion in \cite{bansaye_limit_2011,hardyharris3,harrisroberts2011,Marguet1,Marguet2} and on the other hand, the historical processes, as introduced in Dawson and Perkins \cite{dawsonperkinsAMS,perkinsAMS,perkins} and Dynkin \cite{dynkin91}, and then developed in M\'el\'eard and Tran \cite{meleardtran_suphist} (with a correction, see \cite{kliem,chihdr}). \\
 The historical process taken at a time $t$ describes the history of each individual in the population stopped at time $t$. Since  the individual death rate depends on the total size of the population, the historical process cannot be reduced to an accumulation of independent trajectories. Nevertheless, assuming that the initial condition converges to the stationary solution of \eqref{pde}, an important step in our approach is to replace (up to a negligible error that we can control) the nonlinearity  in the stochastic population process (the total number of individuals) by the mass of the stationary distribution. The birth-and-death process with diffusion becomes a branching-diffusion process and computation becomes rather easier. By coupling techniques, we can therefore capture the dynamics of the historical process using  reasoning proper to branching-diffusion processes and we can easily prove in this context formulas  based on the so-called many-to-one formulas describing the distribution of the ancestry (in forward time) of a typical individual in the population living at time $T$, as it is done in a general context in \cite{Marguet2} with a more complicate proof (since more general). Furthermore, the coupling also allows to justify the use of the well known spinal theory to obtain the law of an individual chosen uniformly at random at time $T$.

 The process that we obtain involves the expectation of the number $m(t,x)$ of individuals at time $t$ issued from one individual with trait $x$. This quantity is obtained as expectation of an additive functional of a drifted Brownian motion and can be explicitly computed by tricky arguments based on Girsanov transform and inspired by \cite{fitzsimmons1993markovian}. Note that this computation allows to obtain the explicit value of the solution of
\begin{align}
& \partial_{t} m{(t,x)}=\frac{\sigma^2}{2}\partial^2_{xx}m{(t,x)}-\sigma c\partial_{x}m{(t,x)}+\left (1-\frac{x^2}{2}-\|F\|_1\right )m{(t,x)} \label{eq:m-intro}\\
& m_{0}(x)=1.\nonumber
\end{align}
The many-to-one formula allows to characterize the forward lineage dynamics as obtained from an auxiliary non-homogeneous Markov process. In this specific case, we obtain the exact trajectory of the trait lineage and prove that they are Gaussian at any time. The last step consists in  using the results by Haussmann and Pardoux \cite{haussmannpardoux} on time reversed diffusion processes.
 We prove that the time reversed paths are Ornstein-Uhlenbeck processes attracted by $0$ as stated by Theorem \ref{reversed}. This proves how the genealogical tree is strongly unbalanced in our case, as observed in Fig. \ref{fig:1}.\\

 For alternative points of view, let us mention that there has been a large literature related to our work. First, there has been a considerable amount of studies dealing with simple models of waves advancing a fitness landscape in an asexual reproducing population, starting from the seminal papers  \cite{tsimring_rna_1996,kessler_evolution_1997}, see also \cite{rouzine_solitary_2003} for a similar model with a nonlinear diffusion operator. In these models, the trait is the fitness itself ({\em i.e.} the growth rate per capita) centered by its average on the population, so that new mutants can outcompete the resident population if their fitness is higher than the mean.  We also refer to analogous studies in the absence of deleterious mutations, by \cite{desai_speed_2007} (including experimental evidence supporting the theory), \cite{beerenwinkel_genetic_2007} in the context of oncogenesis, and \cite{park_speed_2010} for a review article. Mathematical results in this direction were obtained in \cite{durrett_traveling_2011}, then \cite{schweinsberg_rigorous_2017}. Several authors also investigated the structure of the genealogies in stochastic models, exhibiting coalescent structures.
 For coalescent processes in modelling genealogies for populations without competition or interaction non-linearities, we refer to \cite{berestycki_recent_nodate} for a review. For directed selection, when the population at latter stages is issued from individuals at the tip of the wave, strongly asymmetric genealogical trees arise (see \cite{brunet_effect_2007, brunet_genealogies_2013, desai_genetic_2013, neher_genealogies_2013,schweinsberg_rigorous_2017-1,berestyckiberestyckischweinsberg}).
 In \cite{lepersbilliardportemeleardtran}, the genealogies in an adaptive dynamics time scale are described with a forward-backward coalescent. For structured populations with competition, other approaches include the look-down processes \cite{donnellykurtz_96,donnellykurtz_99,etheridgekurtz} or the tree-valued descriptions as in \cite{blancasbenitezguflerkliemtranwakolbinger,grevenpfaffelhuberwinter2009,kliemwinter}. Let us emphasize that, here, we focus on typical lineages rather than coalescent analysis. This is left for a future work.\\
{Finally, let us cite other mathematical contributions with spatial displacement and competition local in space (contrary to \eqref{pde0} where it is global in trait) \cite{loarie_velocity_2009, potapov_climate_2004, berestycki_can_2009, alfaro_effect_2017}. However, these studies focus on the ability of the species to keep pace of the climate change, {\em i.e.} the conditions of persistence for the species, rather than on lineages dynamics.\\

 In Section \ref{sec:PDE}, we introduce and study the individual-based stochastic measure-valued process underlying the PDE \eqref{pde}. The stationary solution of this PDE, which will play a central role in what follows, is also carefully detailed. The stochastic processes associated with \eqref{pde} are non-linear because of the competition term. However, when we are close to the equilibrium, a coupling with a linear birth-death process (with a time-varying growth rate) is possible. This coupling holds for the trait distribution at a given time $T$ but also for the historical picture, i.e. for the ancestral paths of the individuals alive at $T$. This is explained in Section \ref{sec:Feynman-Kac}. For the linear birth-death process, we can apply a Feynman-Kac formula. This, together with fine stochastic calculus techniques, allows us to compute the exact solution of \eqref{eq:m-intro}. In Section \ref{sec:spinal}, we use a many-to-one formula together with the expression of $m_t(x)$ obtained previously and the coupling of historical processes to obtain the approximating stochastic differential equation (SDE) satisfied by the ancestral path of an individual chosen at random in the population at a given time $T$. This SDE is non-homogeneous in time but its time-reverse SDE is a simple time-homogeneous Ornstein-Uhlenbeck process.

\section{The partial differential equation and the population process in the moving framework}\label{sec:PDE}

\subsection{The underlying measure-valued stochastic process}\label{sec:measure-valued-sto-pro}

As explained in the introduction, we are interested in the dynamics of the population  density in the moving framework. We have seen that it is given by  \eqref{pde}.
This equation is well posed. Existence of a weak solution will be obtained from the study of the underlying stochastic process and uniqueness by use of the associate mild equation.

\me Let us introduce the stochastic process  associated with Equation \eqref{pde}. On a probability space $(\Omega,\mathcal{F},\mathbb{P})$, we consider a random process $(Z^{K}_{t})_{t\in \mathbb{R}_{+}}$ with values in the set of  point measures  on $\R$, and defined by
\begin{equation}Z^{K}_{t} = {1\over K} \sum_{i\in V^K_t} \delta_{X^i(t)},\label{def:ZK}
\end{equation}
where $V^K_t$ is the set of labels of individuals alive at time $t$ and where $\,X^i(t)$ denotes the position of the $i$-th individual at time $t$. Individual labels can be chosen in the Ulam-Harris-Neveu set $ \mathcal{I} = \cup_{n\in \N}\mathbb{N}^n$ (e.g. see \cite{legall}) where offspring labels are obtained by concatenating the label of their parent with their ranks among their siblings. Note that the size $N^K_t$ of  the population at time $t$   satisfies $N^K_t= |V^K_t|=K \langle Z^{K}_t,1\rangle$, where the brackets are the notation for the integral of the constant function equal to 1 with respect to the measure $Z^K_t(dx)$. More generally, for a finite measure $\mu$ and a positive measurable function $\varphi$, $\langle \mu, \varphi\rangle=\int_{\R}\varphi(x)\mu(dx)$ denotes the integral of $\varphi$ with respect to $\mu$.\\

In the sequel, we will denote by $\mathcal{M}_{F}(\R)$ the set of finite measures on $\R$ equipped with the topology of weak convergence. The process $Z^K$ belongs to $\D(\R_+,\mathcal{M}_{F}(\R))$, the space of left-limited and right-continuous processes with values in $\mathcal{M}_{F}(\R)$, that we equip with the Skorokhod topology (see e.g. \cite{billingsley_convergence_1999}).

When times vary, the process $(Z^{K}_{t})_{t\in \mathbb{R}_{+}}$ defines a Markov process whose transitions are as follows. For an individual at position $x$ in the population of $N$ individuals,
its birth rate is $1$ and its death rate is $x^2/2 + (N-1)/K$. Between the jumps, the positions $X^i(t)$ behave as a drifted Brownian motions $\sigma\,B_{t}- c\,\sigma\,t$ started at their positions after the jump. All individual births and deaths events and the diffusions between jumps are independent but  the  interaction between individuals to survive is modeled at the individual level by the additional death rate $(N-1)/K$.\\

Following  Champagnat-M\'el\'eard \cite{champagnatmeleard}, we can construct the process $Z^K$ as the unique solution of a stochastic differential equation driven by a Poisson point measure and Brownian motions indexed by $\mathcal{I}$.  (see Appendix \ref{app:sde}). From this representation, and using stochastic calculus for diffusions with jumps (e.g. \cite{ikedawatanabe}), we can derive the following moment estimates, proved in Appendix \ref{app:moment}:
\begin{lem}\label{lem:propagation-moment} We assume that the initial condition $Z^{K}_{0}$ satisfies for $\epsilon>0$ that:
\begin{equation}\label{hyp:moment}
\sup_{K\in \N^*}\E\big(\langle Z^K_0,1\rangle^{2+\epsilon} \big)<+\infty\qquad \mbox{ and }\qquad \sup_{K\in \N^*} \E\big( \langle Z^K_0,x^4\rangle^{1+\epsilon}\big)<+\infty.
\end{equation}
Then,  for any $T>0$, we have
\begin{equation}\sup_{K\in \N^*} \E\big(\sup_{t\in [0,T]} \langle Z^K_t,1\rangle^{2+\epsilon}\big)<+\infty\qquad \mbox{ and }\qquad \sup_{K\in \N^*} \E\big(\sup_{t\in [0,T]} \langle Z^K_t,x^2\rangle^{1+\epsilon/2}\big)<+\infty.\label{eq:propagation}\end{equation}
\end{lem}

It is also standard to write the semi-martingale decomposition of the process $(\langle Z^K_t,\varphi\rangle)_{t\in \R_+}$ for a function $\varphi\in \Co^2_b(\R)$, under the assumption \eqref{hyp:moment}:
\begin{multline}
\label{process-moving}
\langle Z^{K}_{t},\varphi\rangle=\langle Z^{K}_0,\varphi\rangle\\
+\int_{0}^{t} \int_{\mathbb{R}}  \left\{ \left(1-\frac{1}{2}x^2-\langle Z^{K}_{s}, 1\rangle\right)\varphi(x)-\sigma c \varphi'(x)+\frac{\sigma^2}{2}\varphi''(x)\right\} Z^{K}_{s}(dx)\, ds+M^{K,\varphi}_{t}, \end{multline}
where the process $\,M^{K,\varphi}\,$ is a square integrable martingale with predictable quadratic variation process given by
\begin{equation}
\label{qv}
\langle M^{K,\varphi} \rangle_{t} = {1\over K} \int_{0}^t \int_{\R} \Big\{  \Big(1+{x^2\over 2} + \langle Z^{K}_{s}, 1\rangle\Big)\varphi^2(x) + {\sigma^2} (\varphi')^2(x) \Big)\Big\} Z^{K}_{s}(dx) ds.
\end{equation}

In the next section we will need a mild version of this equation. To do that, we introduce the semigroup $(P_{t})_{t\in\mathbb{R}_{+}}$ of the process $\sigma B_{t} - c \sigma t$ and we define, for a fixed $t>0$ and for $\varphi\in C^2_{b}(\R)$,
\begin{equation}\psi(s,x) = P_{t-s}\varphi(x).\label{eq:semi-gp}\end{equation}
Using the trajectorial representation of $Z^{K}_{t}$ (cf.\ Appendix \ref{app:sde})  and integrating these functions, we show in Appendix \ref{app:stoch-mild} that:
\begin{equation}
\label{mild-proc}
\langle Z^{K}_{t},\varphi\rangle=\langle Z^{K}_0,P_{t}\varphi\rangle+\int_{0}^{t} \int_{\mathbb{R}}  \left(1-\frac{x^2}{2}-\langle Z^{K}_{s}, 1\rangle\right)P_{t-s}\varphi(x)\  Z^{K}_{s}(dx)\, ds+\mathcal{M}^{K,\varphi}_{t}, \end{equation}
where $\mathcal{M}^{K,\varphi}_{t}$ is a square integrable martingale computed explicitly in Appendix \ref{app:stoch-mild}.

\begin{thm}\label{thm-CV}
Let us assume that the initial condition $(Z^K_0(dx))_{K}$ satisfies \eqref{hyp:moment} and that $(Z^{K}_0(dx))_{K}$ converges in probability (weakly as measures) to the deterministic finite measure $\,\xi_0(dx)$. Let $T>0$ be given.
The sequence of processes $\,(Z^{K}_{t})_{ t\in [0, T]}$  converges in probability and in $\mathbb{L}^2$, in $\D([0,T],\mathcal{M}_{F}(\R))$ to a deterministic continuous function $\,( \xi_{t}, t\le T)$ of $C([0,T],\mathcal{M}_{F}(\R))$,   satisfying  for each $t>0$ that $\langle \xi_{t}, 1+x^2\rangle <+\infty$ which is the unique solution of the weak equation: $\forall \varphi\in C^2_{b}(\R)$,
\begin{equation}
\label{limit-moving}\langle \xi_{t},\varphi\rangle=\langle \xi_{0},\varphi\rangle+\int_{0}^{t} \int_{\mathbb{R}}  \left\{ \left(1-\frac{1}{2}x^2-\langle \xi_{s}, 1\rangle\right)\varphi(x)-\sigma c \varphi'(x)+\frac{\sigma^2}{2}\varphi''(x)\right\} \xi_{s}(dx)\, ds.\end{equation}
More precisely:\begin{equation}
 \label{CV-L2}\lim_{K\to \infty}\E(\sup_{t\le T} |\langle Z^{K}_{t},\varphi\rangle - \langle \xi_{t},\varphi\rangle|^2) = 0.
 \end{equation}
Moreover, for any $t>0$, the measure $\xi_{t} $ is absolutely continuous with respect to Lebesgue measure and its density $f(t,x)$ is solution of \eqref{pde} issued from $\xi_{0}$.
\end{thm}

\begin{proof}We break the proof into several steps.\\

\noindent \textbf{Step 1:} Let us first prove the uniqueness of $\xi$ solution of \eqref{limit-moving}. For a test function $\psi\in~C^{1,2}_{b}(\R_{+}\times\R)$ of $s$ and $x$, we have by standard arguments that:
\begin{multline}
\label{edp-time}
\langle \xi_{t},\psi(t,.)\rangle=\langle \xi_{0},\psi(0,.)\rangle\nonumber\\
+\int_{0}^{t} \int_{\mathbb{R}}  \left\{  \partial_s \psi(s,x)+ \left(1-\frac{1}{2}x^2-\langle \xi_{s}, 1\rangle\right)\psi(s,x)-\sigma c \partial_x \psi(s,x)+\frac{\sigma^2}{2}\partial^2_{xx}\psi(s,x)\right\} \xi_{s}(dx)\, ds \end{multline}

Now,  we define for a fixed $t>0$, for $\varphi\in C^2_{b}(\R)$, the $C^{1,2}_{b}(\R)$ function $\psi^t$ by
$$\psi^t(s,x) = \mathbb{E}_{x}\Big(\varphi(Y_{t-s})\exp\big(-\int_{0}^{t-s}\frac{Y^2_{u}}{2}du\big)\Big),$$
where $Y$ is the drifted Brownian motion $dY_{t}= \sigma(dB_{t}-cdt)$.
Then
\begin{multline} \partial_s(\psi^t)(s,x)+ \Big(1-\frac{1}{2}x^2-\langle \xi_{s}, 1\rangle\Big)\psi^t(s,x)-\sigma c \partial_x(\psi^t)(s,x)+\frac{\sigma^2}{2}\partial^2_{xx}(\psi^t)(s,x)\\
  = \left(1-\langle \xi_{s}, 1\rangle\right)\psi^t(s,x),\end{multline}
since $\mathbb{E}_{x}\big(\varphi(Y_{t-s})\big)$ is  solution of the backward ``heat" equation. Noting that $\psi^t(t,x) = \varphi(x)$,
and  coming back to \eqref{edp-time} with this function, we obtain
$$\langle \xi_{t},\varphi\rangle=\langle \xi_{0}, \psi^t(0,.)\rangle+\int_{0}^{t} \int_{\mathbb{R}}   \Big(1-\langle \xi_{s}, 1\rangle\Big)\psi^t(s,x) \xi_{s}(dx)\, ds.$$

Notice that if $\|\varphi\|_\infty \le 1$, then $\|\psi^t(s,.)\|_{\infty}\le 1$. By a Gronwall argument we easily prove (see for example Fournier and M\'el\'eard \cite{fourniermeleard}) that two solutions in $\Co([0,T],\mathcal{M}_{F}(\R))$ of this equation started with the same initial condition coincide.

Since the transition semi-group $(P_{t})$ of the  process $(Y_{t})$ is absolutely continuous with respect to Lebesgue measure for any $t>0$, we also deduce by using Fubini's theorem that the same property holds for $\xi_{t}$. Then we write
$$\xi_{t}(dx) = f(t,x) dx$$
and the function $f$ is the unique weak solution of \eqref{pde} issued from $\xi_{0}$.

\medskip
\noindent \textbf{Step 2:} The proof of the convergence is obtained by a compactness-identification-uniqueness argument and the tightness is deduced from the uniform moments obtained in Lemma \ref{lem:propagation-moment}. It is postponed in Appendix.\\

\noindent \textbf{Step 3:} Since the sequence of processes is proved to converge in law in $\mathbb{D}([0,T],\mathcal{M}_F(\R))$ to a deterministic function, it also converges in probability. The limit is continuous in time and thus the convergence is also a uniform convergence (see \cite[p. 124]{billingsley_convergence_1999}).
Then we have proved that for any $T>0$, for any  continuous and bounded function $\varphi$, for any $\varepsilon>0$,
$$\lim_{K\to \infty}\P(\sup_{t\le T} |\langle Z^{K}_{t},\varphi\rangle - \langle \xi_{t},\varphi\rangle|>\varepsilon) = 0.$$
Moreover, uniform moment estimates  yield uniform integrability and then we also have \eqref{CV-L2}.
\end{proof}

\subsection{A unique positive stationary distribution}
\label{sec:F}
 For Equation \eqref{pde}, computation is simple and the existence and explicit value of a stationary state are easy to obtain. Uniqueness is more delicate. In the next sections, we will be interested in considering as initial condition $\xi_{0}$ the stationary state of Equation \eqref{pde}.
\begin{prop}\label{prop:stationary}
There exists a unique non zero positive  stationary distribution of \eqref{pde} if and only if \be
\label{persistence}\frac{c^2}{2}+\frac{\sigma}{2} <1.\ee
In this case, the equilibrium   is given by
\begin{equation}
\label{Fgauss}
F(x)=\frac{\lambda}{\sqrt{2\pi\sigma}}\exp\left(-\frac{(x+c)^{2}}{2\sigma} \right),
\end{equation}
with
\begin{equation}\label{lambda}\|F\|_1=\lambda=1-\frac{c^2}{2}-\frac{\sigma}{2}.\end{equation}
\end{prop}

Let us note that under Condition \eqref{persistence}, the population will persist in long time and its long time density admits a mode in $-c$. This value differs from the optimal trait $0$, which can be interpreted as a lag in the adaptation to environmental change. See Fig. \ref{fig:densiteF}. Indeed, in long time, the solution $u$ of \eqref{pde0} behaves as $F(x-c\sigma t)$ optimal at $-c+c\sigma t$.

\begin{figure}[!ht]
\begin{center}
 \includegraphics[width=9cm,height=6cm]{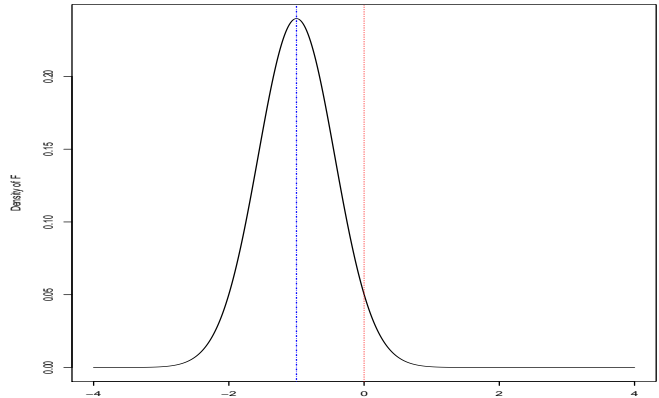}
\caption{{\small Density of the stationary solution $F$ of \eqref{pde}. The mode of this density is $-c=-1$ here, and the variance of $F$ is $\sigma=0.32$, so that $\lambda=0.34$.}
\label{fig:densiteF}}
\end{center}
\end{figure}

\me
\begin{proof} The announced proposition can be obtained from general results in the literature, as the ones of Cloez and Gabriel \cite{cloezgabriel}. We give here a simple proof. Deriving twice  the function $F$ defined in \eqref{Fgauss} and replacing in \eqref{pde}  proves that
\be
\label{Stationary}
\frac{\sigma^2}{2} F''+c\sigma F'+\left(1-\frac{1}{2}x^2-\int_{\mathbb{R}}F(y)dy  \right)F = 0,
\ee
if and only if $\lambda = 1 - {c^2\over 2}-{\sigma\over 2}$. Then under this condition, $F$ is a stationary state.

Let us define the operator $A$ on $C_{b}^2(\R)$ by
$$A\phi =- \frac{\sigma^2}{2}\phi'' - c\sigma \phi'+\frac{1}{2}x^2 \,\phi,$$
for $\phi\in \Co_b^2(\R,\R)$.
Then $F$ is solution of $AF = \alpha F$, with $\alpha = 1-\int F $. Let us also notice that if $\widetilde F$ is defined as $F$ but replacing in \eqref{Fgauss} $c$ by $-c$, then $A^* \widetilde F = \alpha \widetilde F$. Let us now consider  $(\mu, \varphi)$ a solution of  $A \varphi = \mu\, \varphi$ for a positive function $\varphi$ satisfying $\|\varphi\|_{1} = \lambda$.
Then we have
$$\int A \varphi \widetilde F - \int A^*\widetilde F \varphi = 0 = \int (\alpha-\mu) \varphi \widetilde F,$$
with positive functions $\varphi, \widetilde F$ and then $\mu=\alpha$.

Let us now prove that $\varphi = F$. Straightforward computation with  $\psi= \sqrt{F\,\varphi}$ yields
$$A\psi - \alpha \psi= -{\sigma^2\over 8} \frac{(\varphi' F - \varphi F')^2}{\psi^{3}}\le 0.$$
Further, we note that
$$\int A \psi \widetilde F = \int \psi A^* \widetilde F =  \alpha \int \psi \widetilde F.$$
Then , if $D= -{\sigma^2\over 8} \frac{(\varphi' F - \varphi F')^2}{\psi^{3}}$, then
$\int D \widetilde F = 0$, with $D \widetilde F\le 0$ and we deduce that $D=0$ (since $D$ is continuous). Then  we obtain
$\varphi' F  = \varphi F'$ and finally  that $ (\sqrt{{\varphi\over F}})' = 0$, which  implies that $\varphi$ and $F$ are proportional. Since they are both positive with the same $\mathbb{L}^1$ norm, they are equal.

\end{proof}

The next corollary is then an obvious consequence of \eqref{CV-L2}.
\begin{cor}
\label{initialisation-F} Let us assume that  the initial measures $Z^{K}_{0}$ converge weakly to $F(x)dx$ when $K$ tends to infinity, then for any  continuous and bounded function $\varphi$
 \begin{equation}
 \label{CV-F}\lim_{K\to \infty}\E(\sup_{t\le T} |\langle Z^{K}_{t},\varphi\rangle - \langle F,\varphi\rangle|^2) = 0.\end{equation}
\end{cor}

\section{Feynman-Kac approach for an auxiliary branching-diffusion process}\label{sec:Feynman-Kac}
\subsection{Coupling of the process $Z^K$ with a branching-diffusion process}\label{sec:coupling}

  Let us assume in all what follows that  the initial measures $Z^{K}_{0}$ weakly converge to $F(x)dx$ when $K$ tends to infinity as in the Assumption of Corollary \ref{initialisation-F}.

 \me As explained in introduction, we are interested in capturing the genealogies of our particle system. Recall that the ancestral lineage or past history of an individual living at time $T$ consists in the succession of ancestral traits: it is obtained by the concatenation of the (diffusive) paths of this individual with the path of their parent before their birth, then with the path of their grand-parent before the birth of their parent etc. To sum up, the lineage of an individual alive at time $T$ is the path that associates with each time $t\leq T$ the trait of its most recent ancestor at this time. Because of the interactions between individuals, the shape of the lineages of living individuals reflects the competition terms in the past, with lineages that might be extinct. Thus, obtaining an equation describing the ancestry of a ``typical individual'' chosen at random in the population at $T$ is difficult to obtain. See for example the developments of Perkins \cite{perkinsAMS} but with assumptions that exclude logistic competition or see the attempts in \cite{meleardtran_suphist}. Corollary \ref{initialisation-F} suggests us to replace the interaction logistic term $ \langle Z^{K}_{t},1\rangle$ by the constant $\|F\|_1 = \int F(x)dx$. The new process is a much more tractable  branching particle system.

 \me
Therefore we couple $Z^{K}$ with an auxiliary measure-valued process $(\widetilde{Z}^{K}_{t})_{t\ge 0}$,  started from the same initial condition  $Z_{0}^K$ and with the same transitions, except that  the logistic term is frozen at $\lambda=\|F\|_1$ (see Appendix \ref{app:sde}).\\

For the auxiliary process, \eqref{process-moving} becomes, for any $\varphi\in C^2_{b}(\mathbb{R})$,
\begin{equation}
\label{ref:partSys}
\langle\widetilde{Z}^{K}_{t},\varphi\rangle=\langle {Z}^{K}_{0},\varphi\rangle+ \int_0^t \int_{\R} \left\{ \left(1-\frac{1}{2}x^2-\lambda\right)\varphi(x)-c\sigma \varphi'(x)+\frac{\sigma^2}{2}\varphi''(x)\right\}\widetilde{Z}^{K}_{s}(dx) \ ds+\widetilde M^{K,\varphi}_{t},
\end{equation}
where $\widetilde M^{K,\varphi}$ is a square integrable martingale with predictable quadratic variation
\be
\label{qvAux}
\langle \widetilde M^{K,\varphi}\rangle_{t}= {1\over K} \int_{0}^t  \int_{\R} \Big\{  \Big(1+{x^2\over 2} + \lambda)\varphi^2(x) + {\sigma^2} (\varphi')^2(x) \Big)\Big\} \widetilde Z^{K}_{s}(dx) ds.
\ee

Let us remark that with the same arguments as in Theorem \ref{thm-CV}, we can prove that  for any $T>0$,  the measure-valued process
$\,\widetilde{Z}^{K}\,$ converges in $\mathbb{D}([0,T],\mathcal{M}_F(\R))$ uniformly and in probability  to the unique  weak solution $(\widetilde \xi_{t}, t\ge  0)$  of
 \begin{equation}
\label{moving-lin}
\partial_{t}\widetilde \xi(t,x) =\frac{\sigma^2}{2}\Delta \widetilde \xi(t,x)+c\sigma \partial_{x}\widetilde \xi(t,x)+\left(1-\frac{1}{2}x^2- \lambda \right)\widetilde \xi(t,x).
\end{equation}
{starting from the initial data $\widetilde \xi(0,x) = F(x)$.}
By an analogous argument as previously, the measure $\widetilde \xi_{t}$ has a density for any $t>0$ whose uniqueness is classical. Further $F$ is also its unique positive stationary distribution (with given norm).
We also have a similar convergence as in \eqref{CV-F}: for any  continuous and bounded function $\varphi$,
\begin{equation}
 \label{CV-L2tilde}\lim_{K\to \infty}\E(\sup_{t\le T} |\langle \widetilde Z^{K}_{t},\varphi\rangle - \langle F,\varphi\rangle|^2) = 0.\end{equation}

\me As an immediate corollary, we can couple the  process $Z^K$ and the branching-diffusion process  $ \widetilde Z^{K}$, using that $|\langle \widetilde Z^{K}_{t},\varphi\rangle - \langle Z^K_t,\varphi\rangle|\leq |\langle \widetilde Z^{K}_{t},\varphi\rangle - \langle F,\varphi\rangle|+|\langle  Z^{K}_{t},\varphi\rangle - \langle F,\varphi\rangle|$.

\begin{prop}
\label{proximity}
Assume that the initial conditions $(Z^K_0)_{K}$ satisfy \eqref{hyp:moment} and that $\,{Z}^{K}_{0}\xrightarrow[K\to\infty]{w} F\,$. Then for any  continuous and bounded function $\varphi$,
$$ \lim_{K\to+\infty}\E(\sup_{t\le T} |\langle Z^{K}_{t},\varphi\rangle-\langle\widetilde{Z}^{K}_{t},\varphi\rangle|^2) = 0.$$
\end{prop}

We can now work with  the process $(\widetilde{Z}^{K}_{t})_{t\ge 0}$. The main improvement with this process is that the nonlinearity has been tackled. Therefore the process satisfies the branching property and we are authorized  to use some classical tools for these processes.

\subsection{Coupling of the historical processes}\label{section:coupling_historique}

Until now, we described the evolving distributions of the trait, but the individual dimension is lost when the population becomes large. In the sequel, we will also investigate the large population dynamics of the historical processes, which at a time $t$ describes the trait ancestry of individuals alive at that time. Recall the definition of the lineage of an individual given in Section \ref{sec:measure-valued-sto-pro}. For an individual $i\in V^K_T$, let us define their lineage $X^i$. In \eqref{def:ZK}, the labels are taken in the Ulam-Harris-Neveu set $\mathcal{I}$ and we can define by $\preceq$ the usual partial order on $ \mathcal{I}$: $j \preceq i$  means that $j\in \mathcal{I}$ is the ancestor of $i\in \mathcal{I}$ i.e. that there exists $k\in \mathcal{I}$ such that $i=(j,k)$, the concatenation of the labels $j$ and $k$. If the individual $i\in {V}^K_{T}$ was living at time $t$, then ${X}^i(t)$ still denotes the position of $i$ at $t$. But if the individual $i$ was not born at time $t$, then, ${X}^i(t)={X}^j(t)$ where $j\prec i$ is the most recent ancestor of $i$ living at $t$.\\
Since an offspring inherits their parent's trait at birth, and since the trait evolves continuously according to a diffusion during an individual's life, such lineage is a continuous function on $\mathbb{R}$. The path is extended after time $t$ by the trait value at time $t$, so that this continuous function can be defined from $\R$ to $\R$ (and not from $[0,t]$ to $\R$).\\

\medskip

Here, we will adopt the approach developed in Dawson and Perkins \cite{dawsonperkinsAMS} or M\'el\'eard and Tran \cite{meleardtran_suphist}.
Let us define the historical process $H^K$ as the following càdlàg process with values in $\mathcal{M}_F(\cal{C}(\R_+,\R))$:
\begin{equation}H^K_t=\frac{1}{K} \sum_{i\in V^K_t} \delta_{X^i_{.\wedge t}},\label{def:HKt}\end{equation}
where $(X^i_{s\wedge t},\ s\in \R_+)$ is the lineage of the individual $i\in V^K_t$.
To investigate the asymptotic behavior of this process, we introduce (as in Dawson \cite{dawson} or Etheridge  \cite{etheridgebook}) the class of test functions on paths of the form: $\forall y\in \cal{C}(\R_+,\R)$,
\begin{equation}
\label{test-function}
 \varphi(y) = \prod_{j=1}^m g_{j}(y_{t_{j}}),
 \end{equation}
for $m\in \mathbb{N}^*$, $0=t_0\le t_{1}<\cdots <t_{m}$ and $\forall j \in \{1,\cdots,m\}$, $g_{j}\in \Co_{b}^2(\mathbb{R}, \mathbb{R})$. As proved in \cite{dawson}, this class is convergence determining.

\medskip
If $y$  is a continuous path stopped at time $t$ (as the trajectories chosen according $H^K_{t}(dy)$), then
$$\varphi(y)=\prod_{j=1}^m g_{j}(y_{t_{j}\wedge t}) =  \sum_{k=0}^{m-1}\ind_{[t_k,t_{k+1})}(t)\Big(\prod_{j=1}^k g_j(y_{t_j}) \, \prod_{j=k+1}^m g_j(y_t)\Big).$$
and, we introduce:
\begin{equation}
\label{derivee}
\widetilde{D}\varphi(t,y)=  \sum_{k=0}^{m-1}\ind_{[t_k,t_{k+1})}(t)\Big(\prod_{j=1}^k g_j(y_{t_j}) \big(\prod_{j=k+1}^m g_j\big)'(y_t)\Big)
\end{equation}
and
\begin{equation}
\label{laplacien}
\widetilde{\Delta}\varphi(t,y)=  \sum_{k=0}^{m-1}\ind_{[t_k,t_{k+1})}(t)\Big(\prod_{j=1}^k g_j(y_{t_j}) \Delta\big(\prod_{j=k+1}^m g_j\big)(y_t)\Big)
\end{equation}
With this notation, the next lemma is obtained by a direct adaptation of the results in \cite{champagnatmeleard} and can be founded in Appendix:
\begin{lem}
\label{martingale}
Assume that \[\sup_{K\in\mathbb{N}^{\ast}}\E(\langle H^{K}_{0},1\rangle^{2+\varepsilon}+\langle H^K_0(dy),y_0^2\rangle^{1+\varepsilon}) = \sup_{K\in\mathbb{N}^{\ast}}\E(\langle Z^{K}_{0},1\rangle^{2+\varepsilon}+\langle Z^K_0,x^2\rangle^{1+\varepsilon})<+\infty.\,\]
For $\varphi$ defined in \eqref{test-function},
\begin{align}
   \langle H^K_t,\varphi\rangle = & \langle H^K_0,\varphi\rangle + \int_0^t \int_{\Co(\R_+,\R)} \Big(
\frac{\sigma^2}{2} \widetilde{\Delta}\varphi(s,y)-\sigma c \widetilde{D}\varphi(s,y) \nonumber\\
+ & \big( 1- \frac{y_s^2}{2} - \langle H^K_s,1\rangle\big)\varphi(y)\Big)H^K_s(dy)\, ds+    \mathcal{M}^{K,\varphi}_t,\label{martingalecasdawson}
\end{align}
where $\mathcal{M}^{K,\varphi}_t$ is a square integrable martingale with predictable quadratic variation process:
\begin{align}\label{eq:crochet}
    \langle \mathcal{M}^{K,\varphi}\rangle_t= \frac{1}{K} \int_0^t \int_{\Co(\R_+,\R)}  \Big( \big( 1+ \frac{y_s^2}{2} +\langle H^K_s,1\rangle\big)\varphi^2(s,y)  { +  \sigma^2 (\widetilde{D}\varphi(s,y) )^2}
    \Big) H^K_s(dy)\, ds.
\end{align}
\end{lem}
We extend here the mild formula \eqref{mild-proc}.
For $t>0$ fixed, for $m\in \mathbb{N}^*$, $0\le t_{1}<\cdots <t_{m}$ and $\forall j \in \{1,\cdots,m\}$, $g_{j}\in C_{b}^2(\mathbb{R}, \mathbb{R})$, we define for $0\le s< t$ and $y\in \Co(\R_+,\R)$  a  generalized version of the semigroup as
 \begin{equation}
\psi^t(s,y) =  \sum_{k=0}^{m-1}\ind_{[t_k,t_{k+1})}(s)\Big(\prod_{j=1}^k g_j(y_{t_j}) S^{t_{k+1} \wedge t}\big(  \prod_{j=k+1}^m g_j\big)(s,y_s)\Big),
 \end{equation}
 where $S^t(g)(s,x)= \mathbb{E}_{x}\big(g(\widetilde{Y}_{t-s})\exp(-\int_{0}^{t-s} \frac{\widetilde{Y}_{u}^2}{2}du)\big)$ and $\widetilde{Y}_{t}= x+\sigma(B_{t}-ct)$.
  Note that $\psi^t(t,y)=\varphi(y)$ and that
  \[ -\frac{y_{s}^2}{2} \psi^t(s,y)+ \partial_s \psi^t(s,y)+ \frac{\sigma^2}{2} \widetilde{\Delta}_y\psi^t(s,y)-\sigma c \widetilde{D}_y\psi^t(s,y)=0.\]
  The next lemma follows from this property and from Appendix \ref{app:sde-hist} (see \eqref{eq:sde} and \eqref{eq:mart}).

\begin{lem} Under the same assumptions as in Lemma  \ref{martingale},
\begin{multline}
   \langle H^K_t,\varphi\rangle = \langle H^K_0(dy),\psi^t(0,y)\rangle
    +  \int_0^t \int_{\Co(\R_+,\R)}  \big( 1 - \langle H^K_s,1\rangle\big)\psi^t(s,y) H^K_s(dy)\, ds+  R_{t,t}^{K}\label{martingale:mild}
\end{multline}
where for $R^K_{u,t}$ is defined for $u\leq t$ by:
  \begin{multline}
  R^K_{u,t} =  \frac{1}{K} \int_0^u  \sum_{i\in V^K_{s}} \sigma \partial_x \psi^t(s,X^i_{(.\wedge s)}) dB^i_s
  +  \frac{1}{K}\int_0^u \int_{\mathcal{I}}\int_{\R_+}\ind_{\{i\in V^K_{s-}\}} \Big[ \psi^t(s,X^i_{(.\wedge s)})\ind_{\{\theta\leq 1\}} \\ -\psi^t(s,X^i_{(.\wedge s)}) \ind_{\{1<\theta\leq 1+\frac{(X^i_{s})^2}{2}+\langle H_{s-}^{K},1\rangle\}} \Big] \tilde N(ds, di, d\theta).\label{def:Rut}
  \end{multline}
  For any $T>0$, there exists a positive constant $C_{T}$ such that for all $0\leq s\leq t\leq T$
 \begin{equation}\mathbb{E}\big(\sup_{u\le t}(R^K_{u,t})^2\big) < \frac{C_{T}}{K}.\label{estimee:R}\end{equation}
\end{lem}
Note that the process $t\mapsto R^K_{t,t}$ is not a local martingale, as seen in the proof.\\

\begin{proof}
Using Lemma \ref{lem:propagation-moment} and noticing that $\langle H^K_s,1\rangle =\langle Z^K_s,1\rangle$ and $\langle H^K_s,y_s^2\rangle =\langle Z^K_s,x^2\rangle$, we have for any $T>0$ that
\begin{equation}\label{etape19}\sup_{K\geq 1} \E\Big(\sup_{t\in [0,T]} \langle H^{K}_{t},1\rangle^{2+\varepsilon}+\langle H^K_t(dy),y_t^2\rangle^{1+\varepsilon} \Big)<+\infty.\end{equation}
Using \eqref{eq:sde} in appendix, we can write \eqref{martingale:mild}-\eqref{def:Rut}.
 The process $t\mapsto R^K_{t,t}$ is not a martingale, but the process $u\mapsto R^K_{u,t}$, defined for $u\le t$,
  is a martingale. Then we can apply Doob's inequality and write
\begin{multline*}
\mathbb{E}\big(\sup_{u\le t}(R_{u,t}^{K})^2\big)\\ \le  \frac{1}{K}\mathbb{E}\bigg(\ \int_0^t \int_{\Co(\R_+,\R)}  \bigg(\big( 1 + \frac{y_s^2}{2} + \langle H^K_s,1\rangle\big)(\psi^t)^2(s,y) +  \sigma^2 (\partial_x \psi^t(s,y) )^2\bigg)H^K_s(dy)\, ds\bigg).\end{multline*}
The function $\varphi$ defining $\psi^t$ being bounded, we can conclude with \eqref{etape19}.\end{proof}

\medskip

As in the previous section, we can freeze the nonlinearity in the competition term to $\lambda$ and couple the historical process $H^K$ with the historical process $\widetilde{H}^K$ associated with the process $\widetilde{Z}^K$ (this coupling can be done using the same  Poisson point measures, Brownian motions and initial condition as $Z^K$ and $\widetilde{Z}^K$, see Appendix \ref{app:sde-hist}). 

\medskip \begin{prop}\label{prop:coupligHHtilde}
Assume that \eqref{hyp:moment} hold and that $\,{Z}^{K}_{0}\xrightarrow[K\to\infty]{w} F\,$. Then for any  continuous and bounded function $\varphi$ of the form \eqref{test-function},
$$ \lim_{K\to+\infty}\E(\sup_{t\le T} |\langle H^{K}_{t},\varphi\rangle-\langle\widetilde{H}^{K}_{t},\varphi\rangle|^2) = 0.$$
\end{prop}

\begin{proof}
Using Appendix \ref{app:sde-hist}, we have for $\widetilde{H}^K$ a similar  decomposition as \eqref{martingale:mild} with $\langle H^{K}_{t},1\rangle$ replaced by $\lambda$. We now use the mild equation \eqref{martingale:mild} for $H^K$ and the analogous equation for $\widetilde{H}^K$, involving a term $\widetilde{R}^K_{t,t}$ similar to \eqref{def:Rut} but with again $\langle H^{K}_{t},1\rangle$ replaced by $\lambda$. Recall that both processes are built on the same probability space with the same initial condition. For $M>0$, let us introduce the stopping times:
\begin{align}
& \tau^K_M=\inf\big\{t\in \R_+,\ \langle H^K_t,1\rangle =\langle Z^K_t,1\rangle > M \big\},\nonumber\\
\mbox{ and } &
\widetilde{\tau}^K_M=\inf\big\{t\in \R_+,\ \langle \widetilde{H}^K_t,1\rangle =\langle \widetilde{Z}^K_t,1\rangle > M\big\}.\end{align}
Let $\varepsilon>0$ be fixed. Because the processes $(\langle Z^K_t,1\rangle)_{t\in \R_+}$ and $(\langle \widetilde{Z}^K_t,1\rangle)_{t\in \R_+}$ converge to deterministic continuous and bounded processes, there exists $M=M(\varepsilon)$, independent from $K$, such that
\[\P(\tau^K_M \wedge \widetilde{\tau}^K_M\leq T)<\varepsilon.\]
Then, using \eqref{martingale:mild} for a bounded cylindrical test-function $\psi^t$ as in \eqref{test-function}:

\begin{multline*}
\big|\langle H^{K}_{t\wedge \tau^K_M\wedge \widetilde{\tau}^K_M},\varphi\rangle-\langle\widetilde{H}^{K}_{t\wedge \tau^K_M\wedge \widetilde{\tau}^K_M},\varphi \rangle \big|  \\
\begin{aligned}
 \leq &  \Big|\int_0^{t\wedge \tau^K_M\wedge \widetilde{\tau}^K_M} \int_{\Co(\R_+,\R)}  \big( 1-  \langle H^K_s,1\rangle\big)\psi^t(s,y)\ H^K_s(dy)\, ds   \\
& -  \int_0^{t\wedge \tau^K_M\wedge \widetilde{\tau}^K_M} \int_{\Co(\R_+,\R)}  \big( 1- \lambda\big)\psi^t(s,y)\widetilde{H}^K_s(dy)\, ds    \Big| + \big| R_{t\wedge \tau^K_M\wedge \widetilde{\tau}^K_M,t}^{K} \big| + \big|\widetilde{R}_{t\wedge \tau^K_M\wedge \widetilde{\tau}^K_M,t}^K\big|\\
\leq & |1-\lambda| \int_0^t \sup_{u\leq s\wedge \tau^K_M\wedge \widetilde{\tau}^K_M} \big|\langle H^K_{u}(dy) - \widetilde{H}^K_{u}(dy),\psi^t(u,y)\rangle  \big| \ ds \\
  & +  \int_0^{t\wedge \tau^K_M\wedge \widetilde{\tau}^K_M} \big|\langle Z^K_s,1\rangle- \lambda \big|\times \big|\langle H^K_s(dy),\psi^t(s,y)\rangle\big| \ ds   + \sup_{s\leq t} \big| R_{s\wedge \tau^K_M\wedge \widetilde{\tau}^K_M,t}^{K} \big| + \sup_{s\leq t}\big|\widetilde{R}_{s\wedge \tau^K_M\wedge \widetilde{\tau}^K_M,t}^K\big|\\
\leq & |1-\lambda| \|\varphi\|_\infty \int_0^t \sup_{u\leq s} \|  H^K_{u\wedge \tau^K_M \wedge \widetilde{\tau}^K_M} - \widetilde{H}^K_{u\wedge \tau^K_M \wedge \widetilde{\tau}^K_M}\|_{TV} \ ds \\
  & +  T\|\varphi\|_\infty M \ \sup_{s\leq T} \big|\langle Z^K_s,1\rangle- \lambda \big|    + \sup_{s\leq t} \big| R_{s\wedge \tau^K_M\wedge \widetilde{\tau}^K_M,t}^{K} \big| + \sup_{s\leq t}\big|\widetilde{R}_{s\wedge \tau^K_M\wedge \widetilde{\tau}^K_M,t}^K\big|,
\end{aligned}
\end{multline*}where $\|.\|_{TV}$ denotes the norm in total variation. Taking the supremum with respect to $\varphi$ of the form \eqref{test-function} and with norm $\|\varphi\|_\infty\leq 1$ in the left hand side, and then taking the expectation, we have by \eqref{estimee:R}:
\begin{multline*}
\E\big(\sup_{s\leq t} \|H^K_{s\wedge \tau^K_M\wedge \widetilde{\tau}^K_M}-\widetilde{H}^K_{s\wedge \tau^K_M\wedge \widetilde{\tau}^K_M} \|_{TV}\big)\leq  |1-\lambda| \int_0^t \E\Big(\sup_{u\leq s} \|H^K_{u\wedge \tau^K_M \wedge \widetilde{\tau}^K_M}- \widetilde{H}^K_{u\wedge \tau^K_M \wedge \widetilde{\tau}^K_M}\|_{TV} \Big) ds \\
 + TM \ \E\big(\sup_{s\leq T} \big|\langle Z^K_s,1\rangle- \lambda \big|\big)+ 2 \sqrt{\frac{C_T }{K}} .
\end{multline*}
Using Gronwall's lemma:
\begin{equation*}
\E\big(\sup_{s\leq t} \|H^K_{s\wedge \tau^K_M\wedge \widetilde{\tau}^K_M}-\widetilde{H}^K_{s\wedge \tau^K_M\wedge \widetilde{\tau}^K_M} \|_{TV}\big)\leq  \Big(TM\  \E\big(\sup_{s\leq T} \big|\langle Z^K_s,1\rangle- \lambda \big|\big)+ 2 \sqrt{\frac{C_T }{K}} \Big) e^{|1-\lambda| T}.
\end{equation*}Then, note that
\begin{multline}
\mathbb{E}\big(\sup_{t\leq T}\| H^{K}_{t}-\widetilde{H}^{K}_{t} \|_{TV}\big)\leq
\mathbb{E}\big(\sup_{t\leq T}\| H^{K}_{t\wedge \tau^K_M\wedge \widetilde{\tau}^K_M}-\widetilde{H}^{K}_{t\wedge \tau^K_M\wedge \widetilde{\tau}^K_M} \|_{TV}\big) \\
\begin{aligned}
 & + \sqrt{\P(\tau^K_M \wedge \widetilde{\tau}^K_M\leq T)} \sqrt{2  \E\big(\sup_{t\leq T}\langle Z^K_t,1\rangle^2 + \langle \widetilde{Z}^K_t,1\rangle^2\big)}\\
\leq & \Big(TM\  \E\big(\sup_{t\leq T} \big|\langle Z^K_t,1\rangle- \lambda \big|\big)+ 2 \sqrt{\frac{C_T }{K}} \Big) e^{|1-\lambda| T}+\sqrt{\varepsilon}  \sqrt{2 \sup_{K\geq 1}\E\big(\sup_{t\leq T}\langle Z^K_t,1\rangle^2 + \langle \widetilde{Z}^K_t,1\rangle^2\big)},
\end{aligned}\end{multline}for $M>M_\varepsilon$. Then, choosing $K$ sufficiently large, the first term in the right hand side is upper bounded by a constant times $\varepsilon$, by \eqref{CV-L2tilde}. This, with \eqref{eq:propagation}, concludes the proof.
\end{proof}

\subsection{Feynman-Kac approach for the law of the branching-diffusion  process}

As the process   $(\widetilde{Z}^{K}_{t},\ t\le T)$ is a branching process without interaction,  the genealogies started from the  initial individuals evolve independently from each other, with the same law. It follows that
\[
\mathbb{E}\left[\langle \widetilde{Z}_{t}^{K},\varphi\rangle\right]=\int_{\mathbb{R}} \mathbb{E}_{\delta_x}\left[\langle \widetilde{Z}_{t},\varphi\rangle \right] {Z}^{K}_{0}(dx).
\]
where $\widetilde{Z}$ is a branching process satisfying Equation \eqref{ref:partSys} started from $\widetilde{Z}_{0}=\delta_{x}$.
 For the reasons mentioned above, we consider from this point a particle system starting from a single particle with trait $x$.

\medskip \noindent  The formulas/theory used below come from \cite{kurtz1997conceptual,lyons1995conceptual} further developed for instance in \cite{bansaye_limit_2011,cloez,hardyharris3,Marguet2}. Here we give original and  simpler proofs.

\begin{lem}
	\label{lem:MTO}
	Let $\varphi$ in $C_{b}(\mathbb{R})$. Then, for any positive time $t$, for any $x\in \mathbb{R}$, we have
\begin{equation}
\label{eq:MTO}
\mathbb{E}_{\delta_{x}}\left[\langle \widetilde{Z}_{t}, \varphi\rangle \right]=\mathbb{E}_{x}\left[\exp\left(\int_{0}^{t} \left(1-\frac{1}{2}X_{s}^{2}-\lambda \right)ds \right)\varphi(X_{t}) \right],
\end{equation}
where $X$ is the drifted Brownian motion
\begin{equation}\label{def:Y} dX_{t}= \sigma(dB_{t}-c\ dt).\end{equation}

\end{lem}

\begin{proof}
Let us give a very simple proof based on It\^o's formula.

\medskip \noindent  Let us first note that the measure $\,
	\nu_{t}(dy)=\mathbb{E}_{\delta_{x}}\Big(\widetilde{Z}_{t}(dy)\Big)\,$ defined for any $\varphi$ in $C_{b}(\mathbb{R})$ by
	
\[
	\langle\nu_{t}, \varphi \rangle =\mathbb{E}_{\delta_{x}}\left[\langle \widetilde{Z}_{t}, \varphi\rangle \right]
	\]
is  the unique weak solution of
\begin{equation}
	\label{eq:prob2}
	\begin{cases}
	\partial_{t}\nu_{t}=\Delta \nu_{t} + \sigma c \partial_{x}\nu_{t}+(1-\frac{x^2}{2}-\lambda)\nu_{t}\\
	\nu_{0}=\delta_{x}
	\end{cases}.
	\end{equation}
	
	Indeed, it is enough to take expectation in \eqref{ref:partSys}. Uniqueness of such a solution  has been proved in Theorem \ref{thm-CV}.
	
\medskip \noindent Let us now show that the right hand side term of \eqref{eq:MTO} also  satisfies \eqref{eq:prob2}. Uniqueness will yield the result.

Let $\varphi$ in $C^2_{b}(\mathbb{R})$ and apply It\^o's formula to the semimartingale
$$\exp\left(\int_{0}^{t} \left(1-\frac{1}{2}X_{s}^{2}-\lambda \right)ds \right)\varphi(X_{t}) .$$
We have
\begin{multline*}\exp\left(\int_{0}^{t} \left(1-\frac{1}{2}X_{s}^{2}-\lambda \right)ds \right)\varphi(X_{t})= \varphi(X_{0}) + \int_{0}^t \exp\left(\int_{0}^{s} \left(1-\frac{1}{2}X_{u}^{2}-\lambda \right)du \right)\,\sigma\varphi'(X_{s}) dB_{s} \\
+  \int_{0}^t \exp\left(\int_{0}^{s} \left(1-\frac{1}{2}X_{u}^{2}-\lambda \right)du \right) \,\bigg\{\left(1-\frac{1}{2}X_{s}^{2}-\lambda \right)\,\varphi(X_{s}) + {\sigma^2\over 2}\,\varphi''(X_{s})- c \sigma\,\varphi'(X_{s})\bigg\}ds.
\end{multline*}
Taking the expectation, we obtain that
\begin{multline}
\mathbb{E}_{x}\left[\exp\left(\int_{0}^{t} \left(1-\frac{1}{2} X_{s}^{2}-\lambda \right)ds \right)\varphi(X_{t}) \right] =\varphi(x)
 +  \mathbb{E}_{x}\bigg[\int_{0}^t \exp\left(\int_{0}^{s} \left(1-\frac{1}{2}X_{u}^{2}-\lambda \right)du \right) \\
 \times \bigg\{\left(1-\frac{1}{2}X_{s}^{2}-\lambda \right)\,\varphi(X_{s}) + {\sigma^2\over 2}\,\varphi''(X_{s})- c \sigma\,\varphi'(X_{s})\bigg\}ds\bigg]. \label{ito}\end{multline}

If we define the measure $\mu_{t}$ for any test function $\varphi\in \Co^2_b(\R)$ by
$$\langle \mu_{t}, \varphi\rangle = \mathbb{E}_{x}\left[\exp\left(\int_{0}^{t} \left(1-\frac{1}{2}X_{s}^{2}-\lambda \right)ds \right)\varphi(X_{t}) \right] ,$$
we obtain from \eqref{ito}  that
$$\langle \mu_{t}, \varphi\rangle = \langle \delta_{x}, \varphi\rangle +\int_{0}^t \Big\langle \mu_{s}, (1-{x^2\over 2} -\lambda)\varphi(x) + {\sigma^2\over 2}\,\varphi''(x)- c \sigma\,\varphi'(x)\Big\rangle ds.$$
This proves that the flow $(\mu_{t},\ t\ge 0)$ is a weak solution of \eqref{eq:prob2} and the conclusion follows by uniqueness.
\end{proof}

\begin{cor}
\label{avril}
Let us define  for any $t\ge 0$ and $x\in \mathbb{R}$ the expectation of the number of individuals at time $t$ in the branching process $\widetilde{Z}_t$ started from one individual wih trait $x$,
\begin{equation}
 \label{m}
  m_{t}(x) = \E_{\delta_{x}} (\langle \widetilde{Z}_{t}, 1\rangle).\end{equation}
Then we have
\begin{equation}
\label{tenan}
m_{t}(x)  = \mathbb{E}_{x}\left[\exp\left(\int_{0}^{t} \left(1-\frac{1}{2}X_{s}^{2}-\lambda \right)ds \right) \right].
\end{equation}
We deduce that the function $(t,x)\mapsto m_t(x)$ belongs to $C_{b}^{1,\infty}([0,T]\times \mathbb{R})$.
\end{cor}

\medskip \begin{proof}
Equation \eqref{tenan} is obvious  by  applying Lemma \ref{lem:MTO} to $\varphi=1$.

Since the process $X$ is a drifted Brownian motion,  we can write
$$m_{t}(x) = \mathbb{E}_{0}\left[\exp\left(\int_{0}^{t} \left(1-\frac{1}{2}(x+X_{s})^{2}-\lambda \right)ds \right) \right].$$
Lebesgue's Theorem allows us to conclude.
\end{proof}

\noindent
{\bf Remark:} \emph{From   \eqref{tenan} and Feynman-Kac formula (see \cite{feynman_space-time_1948,kac_distributions_1949,kac_connections_1951}), we deduce that  the function $(m_{t}(x), x\in \mathbb{R}, t\ge 0)$ is the unique strong solution of
\begin{eqnarray}
\label{m-edp}
\begin{cases}
\partial_{t} m=\frac{\sigma^2}{2}\partial_{xx}m-\sigma c\partial_{x}m+(1-\frac{x^2}{2}-\lambda)m\\
m_{0}(x)=1.
\end{cases}
\end{eqnarray}
}

\emph{Let us also note that  \eqref{m-edp} and the stationarity of $F$ (see Eq.\ \eqref{Stationary}) imply  that $t\to \int m_{t}(x) F(x) dx$ is constant and then
\be
\label{mF}
\int_{\mathbb{R}} m_{t}(x) F(x) dx = \int_{\mathbb{R}} F(x) dx = \lambda.
\ee}

\medskip Our aim is now to generalize \eqref{eq:MTO} to trajectories.  In what follows, for a time $T$, we label individuals by $i\in \widetilde{V}_T$ where $\widetilde{V}_{T}$ denotes the set of individuals alive at time $T$ and started from one individual with trait $x$ at time $0$. For a time $t<T$, we will introduce the  notation $\widetilde{X}^i(t)$ to denote the historical lineage of  the individual $i\in \widetilde{V}_{T}$ at time $t$.

\begin{lem}
\label{lem:MTO05}
	Let $\varphi$ in $C_{b}(\mathbb{R})$. Then, for any positive times $t$ and $T$ such that $t\leq T$, for any $x\in \mathbb{R}$, we have
\begin{equation}
	\label{eq:MTO05}
	\mathbb{E}_{\delta_{x}}\left[\sum_{i\in \widetilde{V}_{T}} \varphi(\widetilde{X}^{i}_{t}) \right]=\mathbb{E}_{x}\left[\exp\left(\int_{0}^{T} \left(1-\frac{1}{2}X_{s}^{2}-\lambda \right)ds \right)\varphi(X_{t}) \right],
\end{equation}
where $X$ is the drifted Brownian defined in Lemma \ref{lem:MTO}.
\end{lem}
\begin{proof}The case where $t=T$  results from Lemma \ref{lem:MTO}.
	To obtain formula \eqref{eq:MTO05} from \eqref{eq:MTO} when $t<T$, one can proceed as follows. For every individual $i$ alive at time $T$, and for every $t\leq T$, there exists a unique $j\in \mathcal{I}$ such that $(j,t)$ belongs to the ancestral path of $i$. 	
	Thus, we have
	\[
	\sum_{i\in \widetilde{V}_{T}}\varphi(\widetilde{X}^{i}_{t})=\sum_{i\in \widetilde{V}_{T}}\sum_{j\in \widetilde{V}_{t}} \ind_{(j,t) \preceq (i,T)} \varphi(\widetilde{X}^{j}_{t}),
	\]
	where $	\widetilde{X}^{i}_{t}=\widetilde{X}^{j}_{t}$ since $(j,t)\preceq (i,T)$.
	Thus,
	\[
	\sum_{i\in \widetilde{V}_{T}}\sum_{j\in \widetilde{V}_{t}} \ind_{j \preceq i} \varphi(\widetilde{X}^{i}_{t})=\sum_{j\in \widetilde{V}_{t}}\varphi(\widetilde{X}^{j}_{t})\sum_{i\in \widetilde{V}_{T}} \ind_{j \preceq i} =\sum_{j\in \widetilde{V}_{t}}\varphi(\widetilde{X}^{j}_{t})\widetilde{N}(j)_{t,T},
	\]
	where $\widetilde{N}(j)_{t,T}$ denotes the number of descendents at time $T$ of an individual $j$ alive at time $t$ (with the convention that $\widetilde{N}(j)_{t,T}=0$ if $j\notin \widetilde{V}_{t}$, i.e.\ if $j$ does not exist at time $t$). Thus, denoting by $\left(\mathcal{F}_{t},\ t\in\mathbb{R}_{+} \right)$ the natural filtration associated with $\widetilde Z$, we have that
	\begin{eqnarray}
	\label{florian}
	\mathbb{E}\left[\sum_{i\in \widetilde{V}_{T}}\varphi(\widetilde{X}^{i}_{t})\Bigg| \mathcal{F}_{t}\right]&=&\sum_{j\in \widetilde{V}_{t}}\varphi(\widetilde{X}^{j}_{t})\mathbb{E}\left[\widetilde{N}(j)_{t,T}\bigg| \mathcal{F}_{t}\right]=\sum_{j\in \widetilde{V}_{t}}\varphi(\widetilde{X}^{j}_{t})m_{T-t}(\widetilde{X}^{j}_{t})\nonumber \\
	&=&\langle \widetilde{Z}_{t},\varphi\, m_{T-t}\rangle,
	\end{eqnarray}
	where $m_{t}$ has been defined in Corollary \ref{avril}.
	
	\medskip We  now apply Lemma \ref{lem:MTO} to \eqref{florian} for the function  $x\mapsto \varphi(x)\,m_{T-t}(x)$. That  gives
	\[
	\mathbb{E}_{\delta_{x}}\left[\langle \widetilde{Z}_{t},\varphi \,m_{T-t}\rangle\right]=\mathbb{E}_{x}\left[\exp\left(\int_{0}^{t} \left(1-\frac{1}{2}X_{s}^{2}-\lambda \right)ds \right)\varphi(X_{t})\,m_{T-t}(X_{t}) \right].
	\]
	Then, from the expression  of $m_{t}$ given in \eqref{tenan} and  one obtains by the  Markov property that
	\[
	\mathbb{E}_{\delta_{x}}\left[\langle \widetilde{Z}_{t},\varphi \,m_{T-t}\rangle\right]=\mathbb{E}_{x}\left[\exp\left(\int_{0}^{T} \left(1-\frac{1}{2}X_{s}^{2}-\lambda \right)ds \right)\varphi(X_{t}) \right].
	\]
	That concludes the proof.
\end{proof}

\medskip We are now interested in trajectorial extension of the previous formulae.

\medskip  \begin{prop}
\label{lem:MTO06}
	Let $\varphi$ in $\Co_{b}(\mathbb{R}^{n},\R)$. Then, for any positive times $t_{1}<t_{2}<\ldots<t_{n}<T$, for any $x\in \mathbb{R}$, we have
	\begin{equation}
	\label{eq:MTO06}
	\mathbb{E}_{\delta_{x}}\left[\sum_{i\in \widetilde V_{T}} \varphi(\widetilde X^{i}_{t_{1}},\ldots, \widetilde X^{i}_{t_{n}}) \right]=\mathbb{E}_{x}\left[\exp\left(\int_{0}^{T} \left(1-\frac{1}{2}X_{s}^{2}-\lambda \right)ds \right)\varphi(X_{t_{1}},\ldots,X_{t_{n}})  \right],
	\end{equation}
	where $X$ is the drifted Brownian defined in Lemma \ref{lem:MTO}. \end{prop}
\begin{proof} By usual arguments, it is enough to prove the result for product functions
$$\varphi(x_{1},\cdots, x_{n})= \prod_{i=1}^n \varphi_{i}(x_{i}).$$
The proof of this result follows the same lines as the proof of Lemma \ref{lem:MTO05} conditionning first by $\mathcal{F}_{t_{1}}$, and then by $\mathcal{F}_{t_{2}}$ and so on. We leave the remaining of the proof to the reader.
\end{proof}

Let us recall that $\widetilde{H}$ is the historical process associated with $\widetilde{Z}$.
\begin{lem}
We have that for $T>0$, $\Phi:C([0,T],\mathbb{R})\to\mathbb{R}$ a continuous and bounded function and $x\in \mathbb{R}$:
\begin{multline}
\label{eq:MTO099}
\mathbb{E}_{\delta_x}\left[\langle \widetilde{H}_T,\Phi\rangle\right]=\mathbb{E}_{\delta_{x}}\left[\sum_{i\in \widetilde{V}_{T}}\Phi(\widetilde{X}^i_{s},\ s\leq T) \right]\\
=\mathbb{E}_{x}\left[\exp\left(\int_{0}^{T} \left(1-\frac{1}{2}X_{s}^{2}-\lambda \right)ds \right)\Phi(X_{s},\ s\leq T) \right],
\end{multline}
where $X$ is the drifted Brownian motion defined in Lemma \ref{lem:MTO}.
\end{lem}

\begin{proof}
Let us consider the linear interpolation  $I_{n}:\mathbb{R}^{n}\to C([0,T])$ such that, for all $j\in\{0,\ldots,n-1\}$, for all $t\in[jT/n,(j+1)T/n)$ and $(u_1,\dots u_n)\in \R^{n+1}$,
\[
I_{n}(u_{0},\ldots,u_{n})(t)=(u_{j+1}-u_{j})\frac{n}{T}\Big(t-\frac{j}{n}T\Big)+u_{j}.
\]
Thus, we have for $x\in \Co([0,T],\R)$,
\[
\left\|I_{n}\big(x(0),\ldots,x(jT/n),\ldots , x(T)\big)-x\right\|_\infty\leq 2\, \sum_{j=0}^{n}\ind_{t\in[jT/n,(j+1)T/n)}\omega(x,T/n)= 2\, \omega(x, T/n),
\]
where $\omega(x,.)$ is the modulus of continuity of $x$.
Thus, the functions $I_{n}(x(0),\ldots,x(jT/n) ,\ldots , x(T))$ converge uniformly on $[0,T]$ to $x$, as $n$ tends to infinity. The result then follows from Lebesgue's theorem, Lemma \ref{lem:MTO06} and the continuity of $x$.
\end{proof}

\subsection{Computation of $m_{t}(x)$}
 In this Brownian framework, one can explicitely compute $m_{t}(x)$ from \eqref{tenan}  by  a method adapted from Fitzsimmons Pitman and Yor \cite{fitzsimmons1993markovian}.

\begin{prop}\label{prop:mtx}
For any $x\in \mathbb{R}$ and $t\in [0,T]$, we have
\begin{equation}
\label{def:m}
m_{t}(x)=\sqrt{1+\tanh(\sigma t)}\exp\left(-\frac{\left( x+e^{-\sigma t}c \right)^2}{2\sigma}\left(1+\tanh(\sigma t) \right)+\frac{( x+c)^2}{2\sigma} \right).
\end{equation}
\end{prop}

\begin{proof}
Recall that $X_{t}=\sigma B_{t}-c\sigma t$ (see \eqref{def:Y}). By notational simplicity we assume that  the Brownian motion $B$ starts from $x$, then $X_0=\sigma x$ and we will compute $m_{t}(\sigma x)$. Equation \eqref{tenan} gives
\begin{align}
m_{t}(\sigma x)&=\mathbb{E}_{\delta_{\sigma x}}\left[\langle \widetilde{Z}_t, 1\rangle \right]=\mathbb{E}_{x}\left[\exp\left(\int_{0}^{t} \left(1-\frac{\sigma^2}{2}\left( B_{s}- c s\right)^2 -\lambda \right)ds \right)\right] \nonumber\\
%&=\mathbb{E}_{0}\left[\exp\left(\int_{0}^{t} \left(1-\frac{\sigma^2}{2}\left( B_{s}+x+ c s\right)^2 -\lambda \right)ds \right)\right]\\
 \nonumber\\&=e^{(1-\lambda)t}\mathbb{E}_{x}\left[\exp\left(-\frac{\sigma^2}{2}\int_{0}^{t}  (B_{s}-cs)^2 ds \right)\right]=: e^{(1-\lambda)t} I.\label{m2}
\end{align}
Let us compute explicitly $I$, the expectation appearing in the right hand side of \eqref{m2}. Recall that our probability space is endowed with the probability measure $\mathbb{P}$. Let $(\mathcal{F}^B_{t})_{t\geq 0}$ be the filtration of the Brownian motion $B$ and define the new probability $\mathbb{Q}$ by
\[
\frac{d\mathbb{Q}}{d\mathbb{P}}\mid\mathcal{F}^B_{t}=\exp(cB_{t}-\frac{c^2}{2}t-cx)
\]
using Girsanov theorem to kill the drift.
Under $\mathbb{Q}$, $W_{t}=B_{t}-ct$ is a Brownian motion. Hence,
\begin{align}
I= & \mathbb{E}_{x}^{\mathbb{Q}}\left[\exp\left(-\frac{\sigma^2}{2}\int_{0}^{t}  (B_{s}-cs)^2 ds \right)\exp(-cB_{t}+\frac{c^2}{2}t+cx)\right]\nonumber\\
= & \mathbb{E}_{x}^{\mathbb{Q}}\left[\exp\left(-\frac{\sigma^2}{2}\int_{0}^{t}  W_{s}^2 ds \right)\exp(-cW_{t}-\frac{c^2}{2}t+cx)\right]\nonumber\\
= & e^{cx-\frac{c^2}{2}t}\mathbb{E}_{x}^{\mathbb{Q}}\left[\exp\left(-\frac{\sigma^2}{2}\int_{0}^{t}  W_{s}^2 ds-cW_{t} \right)\right].\label{etape1}
\end{align}
Now, we want to compute the expectation in this last term.
We use that
$M_{t}=\frac{\sigma}{2}(W_{t}^2-t)-\frac{\sigma x^2}{2} = \sigma \int_{0}^t W_{s}\ dW_{s}
$ is a martingale with
\[
\langle M\rangle_{t}=\sigma^2\int_{0}^{t}W_{s}^2\ ds.
\]
Let $(\mathcal{F}^W_{t})_{t\geq 0}$ be the filtration of $W$ and set:
\[
\frac{d\mathbb{Q}'}{d\mathbb{Q}}\mid\mathcal{F}^W_{t}=\exp\left(\frac{1}{2}\sigma W_{t}^2-\frac{1}{2}\sigma t-\frac{\sigma^2}{2}\int_{0}^{t}W_{s}^{2}\ ds\right)e^{-\frac{\sigma x^2}{2}}.
\]
We have
\begin{equation}\label{etape2}
\mathbb{E}_{x}^{\mathbb{Q}}\left[\exp\left(-\frac{\sigma^2}{2}\int_{0}^{t}  W_{s}^2\ ds -cW_{t}\right)\right]=\mathbb{E}_{x}^{\mathbb{Q}'}\left[\exp\left(-\frac{1}{2}\sigma W_{t}^2-c W_{t} \right) \right]e^{\frac{\sigma x^2}{2}+\frac{\sigma t}{2}}.
\end{equation}
On the other hand, we have that under $\mathbb{Q}'$,
\[
W'_t=W_{t}-\sigma\int_{0}^{t}W_{s}\ ds
\]
is a $\mathbb{Q}'$-Brownian motion, and
\[
W_{t}=\sigma\int_{0}^{t}W_{s}\ ds+W'_{t}
\]is an Ornstein-Uhlenbeck process. Hence, for a given $t\geq 0$, $W_t$ under $\mathbb{Q}'$ is distributed as
\[
\mathcal{N}\left(e^{\sigma t}x,\frac{1}{2\sigma}(e^{2\sigma t}-1) \right).
\]
This allows us to compute \eqref{etape2}. For $Y$ a standard Gaussian random variable $\mathcal{N}(0,1)$,
\[
\Psi(u,v):=\mathbb{E}\left[e^{uY^2+vY} \right] =\frac{\exp\left(\frac{v^2}{2(1-2u)} \right)}{\sqrt{1-2u}},
\]
when $u<\frac{1}{2}$. For $Y$ any $\mathcal{N}(m,\delta^2)$ random variable,
\[
\mathbb{E}\left[e^{aY^2+bY} \right]=e^{(bm+m^2a)}\Psi(a\delta^2,\delta(2am+b)),\]
when $a \delta^2 <1/2$.\\

Here we have $m=e^{\sigma t}x$, $\delta^2=(e^{2\sigma t}-1)/(2\sigma)$, $a=-\sigma/2$ and $b=-c$. Thus,
\[
\begin{cases}
a\delta^2=-\frac{1}{4}\left(e^{2\sigma t}-1 \right)\\
\delta(2am+b)=\sqrt{\frac{(e^{2\sigma t}-1)}{2\sigma}}\left(-\sigma x e^{\sigma t}-c \right),
\end{cases}
\]
and applying the above computation yields for the expectation in the r.h.s. of \eqref{etape2}:
\begin{align}
\mathbb{E}_{x}^{\mathbb{Q}'}\left[\exp\left(-\frac{1}{2}\sigma W_{t}^2-c W_{t} \right) \right]= & e^{(-cxe^{\sigma t}-\frac{\sigma}{2}x^2e^{2\sigma t})}\frac{1}{\sqrt{\frac{1}{2}+\frac{1}{2}e^{2\sigma t}}}\exp\left(\frac{\frac{(e^{2\sigma t}-1)}{2\sigma}\left(\sigma x e^{\sigma t}+c \right)^2}{1+e^{2\sigma t}} \right)\nonumber\\
= & \sqrt{\frac{2}{1+e^{2\sigma t}}} e^{-\frac{1}{2\sigma} \big(\sigma x e^{\sigma t}+c\big)^2 } e^{\frac{c^2}{2\sigma}}\exp\left(\frac{(\sigma x e^{\sigma t}+c)^2}{2\sigma}\tanh(\sigma t)\right)\nonumber\\
= & \sqrt{\frac{2}{1+e^{2\sigma t}}}  e^{\frac{c^2}{2\sigma}}\exp\left(-\frac{(\sigma x e^{\sigma t}+c)^2}{2\sigma}\big(1-\tanh(\sigma t)\big)\right)\label{etape5}
\end{align}
where the second line has been obtained by using that
\[
\exp\big(-\frac{1}{2\sigma}(\sigma^2 x^2e^{2\sigma t}+2cx\sigma e^{\sigma t}+c^2-c^2)\big)=\exp\big(-\frac{1}{2\sigma}((\sigma xe^{\sigma t}+c)^2-c^2)\big),
\]
and that
\begin{align*}
\frac{e^{2\sigma t}-1}{e^{2\sigma t}+1}=\tanh(\sigma t).
\end{align*}
Gathering \eqref{etape5} with \eqref{etape1} and \eqref{etape2}, we obtain that:
\begin{align*}
I &=e^{cx-\frac{c^2}{2}t}   e^{\frac{\sigma x^2}{2}+\frac{\sigma t}{2}} e^{\frac{c^2}{2\sigma}}\sqrt{\frac{2}{1+e^{2\sigma t}}}\exp\left(-\frac{\left(\sigma xe^{\sigma t}+c \right)^2}{2\sigma}\left(1-\tanh(\sigma t)\right) \right)\\
%&=e^{-\frac{1}{2}c^2t}e^{\frac{1}{2\sigma}(\sigma x+c)^2}\sqrt{\frac{2}{1+e^{2\sigma t}}}\exp\left(-\frac{\left(\sigma xe^{\sigma t}+c \right)^2}{2\sigma}\left(1-\tanh(\sigma t)\right) \right)\\
&=e^{-\frac{c^2t}{2}+\frac{\sigma t}{2}}\sqrt{\frac{2}{1+e^{2\sigma t}}}\exp\left(\frac{(\sigma x+c)^2}{2\sigma}-\frac{\left(\sigma xe^{\sigma t}+c \right)^2}{2\sigma}\left(1-\tanh(\sigma t)\right) \right).
\end{align*}
Plugging this result into \eqref{m2} gives:
\begin{align*}
m_{t}(\sigma x)
&=e^{(1-\frac{1}{2}c^2+\frac{\sigma}{2}-\lambda)t}\frac{\sqrt{2}}{\sqrt{1+e^{2\sigma t}}}\exp\left(\frac{(\sigma x+c)^2}{2\sigma}-\frac{\left(\sigma xe^{\sigma t}+c \right)^2}{2\sigma}\left(1-\tanh(\sigma t)\right) \right)\\
= & e^{(1-\frac{1}{2}c^2-\frac{\sigma}{2}-\lambda)t}\frac{\sqrt{2}}{\sqrt{1+e^{-2\sigma t}}}\exp\left(\frac{(\sigma x+c)^2}{2\sigma}-\frac{\left(\sigma xe^{\sigma t}+c \right)^2}{2\sigma}\left(1-\tanh(\sigma t)\right) \right)\\
&=\sqrt{1+\tanh(\sigma t)}\exp\left(\frac{(\sigma x+c)^2}{2\sigma}-\frac{\left(\sigma xe^{\sigma t}+c \right)^2}{2\sigma}\left(1-\tanh(\sigma t)\right) \right),
\end{align*}where we use \eqref{lambda} and $2/(1+e^{-2x})=1+\tanh(x)$ for the third equality. Replacing in the above expression $x$ with $x/\sigma$ yields the announced expression for $m_t(x)$.
\end{proof}

\section{A spinal approach - The typical trajectory}\label{sec:spinal}

\subsection{The spinal process}
{Recall  that the stochastic population process is assumed starting} from the stationary  distribution $F$.
We want to characterize the behavior of the ancestral path of an individual uniformly sampled  at time $T$.
To this aim, the spinal approach \cite{hardyharris2,hardyharris3,bansaye_limit_2011} consists in considering the trajectory of a ``typical" individual in the population whose behavior summarizes the behavior of the entire population. The next theorem, will allow to describe the trait process along the spine and can be found in \cite{Marguet2} in a more general context.
To make the paper easy to read, a proof in our context is given in Appendix B.
\begin{thm}
\label{spine}
Recall that $m_t(x)=\E_x(\widetilde{N}_t)$ has been defined in \eqref{tenan}. For $T>0$, $x\in \R$ and $\Phi$ a continuous bounded function on $\Co([0,T],\R)$, we have
\begin{equation}
\label{eq:MTO2}
\mathbb{E}_{\delta_{x}}\left[\sum_{i\in \widetilde{V}_{T}}\Phi(\widetilde{X}^i_{s},\ s\leq T) \right]={m_{T}(x)}\mathbb{E}_{x}\left[\Phi(Y_{s},\ s\leq T) \right],
\end{equation}
where $Y$ is an inhomogeneous Markov process (depending on $t$) and with infinitesimal generator at time $t$ given for $\phi\in C_{b}^2(\mathbb{R})$ by
\begin{equation}\label{eq:generatorG}
\mathcal{G}_t \phi(x)=\frac{L(m_{T-t}\phi)(x)-\phi(x)Lm_{T-t}(x)}{m_{T-t}(x)},
\end{equation}
where $L$ is the infinitesimal generator  of the process $(X_{t}, t\ge 0)$ associated to $\widetilde{Z}$ and defined in \eqref{def:Y}.
\end{thm}

 Note that the law of the spinal process is biased by the population size at each time, described by the function $m$, which makes the process inhomogeneous. This highlights the form of the generator given in \eqref{eq:MTO2}. \\
 The next proposition allows to relate \eqref{eq:MTO2} to the distribution of an individual chosen uniformly at random among the population alive at time $t$ (i.e. in the empirical distribution) when the population is large.

\me  \begin{prop}\label{prop:LGN}
For any $\Phi$ continuous and bounded,
$${\lim_{K\to +\infty}\mathbb{E}\left[\frac{1}{\widetilde N^K_{T}}\sum_{i\in \widetilde V^K_{T}}\Phi(\widetilde{X}^i_{s},\ s\leq T) \ \bigg|\ \widetilde{X}^i_{0} = x, \forall i\in \widetilde V^K_{0}\right]} = {\mathbb{E}_{\delta_{x}}\left[\sum_{i\in \widetilde V_{T}}\Phi(\widetilde{X}^i_{s},\ s\leq T) \right]\over m_{T}(x)},$$
where $\widetilde N^K_{T} = K \langle \widetilde Z^K_{T},1\rangle$ and $\widetilde V^K_{T}$ is the set of individuals alive at time $T$. \end{prop}

\me \begin{proof}
By the branching property, the trees started from each of the $K$ individuals with trait $x$ are independent with same law. Then by the law of large numbers, the two sequences $({\widetilde N^K_{T}\over K})_{K}$ and $({\sum_{i\in \widetilde{V}^K_{T}}\Phi(\widetilde X^i_{s},\ s\leq T)\over K})_{K}$ converge almost surely respectively to $m_{T}(x)$ and $\mathbb{E}_{\delta_{x}}\left[\sum_{i\in \widetilde V_{T}}\Phi(\widetilde X^i_{s},\ s\leq T) \right]$.
The result follows.
\end{proof}

\bigskip Notice that a corollary of Theorem \ref{spine} and Proposition \ref{prop:LGN} is that:
\begin{cor}\label{cor:hist-spine}When $(Z^K_0)$ satisfies \eqref{hyp:moment} and converges weakly and in probability to the measure $\xi_0$ when $K\rightarrow +\infty$,
\begin{eqnarray}
\lim_{K\rightarrow +\infty}\E_{Z^K_0}\left[\frac{\langle \widetilde{H}^K_T,\Phi\rangle}{\langle \widetilde{H}^K_T,1\rangle}\right]&=&\lim_{K\to +\infty}\mathbb{E}\left[\frac{1}{\widetilde N^K_{T}}\sum_{i\in \widetilde V^K_{T}}\Phi(\widetilde{X}^i_{s},\ s\leq T) \ \bigg|\ \sum_{i\in \widetilde V^K_{0}} \delta_{\widetilde{X}^i_{0}}= K Z^K_0 \right]\nonumber  \\
&=&\frac{ \int_{\mathbb{R}} m_{T}(x) \mathbb{E}_{x}\left[\Phi(Y_{s},\ s\leq T)\right]\xi_0(dx)}{\int_{\mathbb{R}} m_{T}(x) \xi_0(dx)},
\label{eq:spine-processuslinearise}
\end{eqnarray}
where $Y$ is the process with generator \eqref{eq:generatorG}.
\end{cor}

The explicit computation of $m_t(x)$ yields the generator of $Y$:
\begin{prop}\label{prop:generatorG-explicite}The generator of the spine $Y$ describing in forward time the path of particle chosen at random in $\widetilde V^K_T$ is given for $x\in \R$ and $0\le t\le T$ by
\begin{equation}\label{eq:generatorG-explicite}
\mathcal{G}_{t}f(x)=\frac{\sigma^2}{2}f''(x)-\sigma x\tanh(\sigma (T-t))f'(x)-\frac{\sigma c}{\cosh(\sigma (T-t))} f'(x).
\end{equation}
\end{prop}

\begin{proof}
We have
\begin{align*}
m_{T-t}(x) \mathcal{G}_{t}f(x)&= {\sigma^2\over 2} \big(m_{T-t} f\big)''(x) - c \sigma \big(m_{T-t} f\big)'(x) - f  \big({\sigma^2\over 2}m_{T-t}'' - c \sigma m_{T-t} ' \big)(x)\\
 &=   {\sigma^2\over 2} \big( 2 m_{T-t}'(x) f'(x) + m_{T-t}(x) f''(x) \big)- c \sigma m_{T-t}(x) f'(x) .
\end{align*}

Since, by a derivation of \eqref{def:m},
\begin{align*}
\partial_{x}m_{t}(x)&=\frac{1}{\sigma}\left(-\left( x +c e^{-\sigma t}\right)(1+\tanh(\sigma t))+ x+c\right)m_{t}(x)\\
%&=\sigma x\left(1-e^{2\sigma t}\left(1-\tanh\left(\sigma t\right)\right)\right)+c\left(e^{\sigma t}(1-\tanh(\sigma t))-1 \right) \\
&=\frac{1}{\sigma}\left(- x\tanh(\sigma t)-c\left(e^{-\sigma t}(1+\tanh(\sigma t))-1 \right)\right)m_{t}(x),
\end{align*}
we obtain that for $x\in \R$ and $0\le t\le T$,
\begin{align*}
\mathcal{G}f(x)&=\frac{\sigma^2}{2}f''(x)+\sigma\left(- x\tanh(\sigma (T-t))- c \left({1\over \cosh(\sigma (T-t))} - 1\right)\right)f'(x)-\sigma cf'(x)\\
&=\frac{\sigma^2}{2}f''(x)- \sigma x\tanh(\sigma (T-t)f'(x) -  \sigma c {1\over \cosh(\sigma (T-t))}f'(x).
\end{align*}
\end{proof}

Let us highlight that
\[
\begin{cases}
\mathcal{G}f\simeq
\frac{\sigma^2}{2}f''-\sigma xf', \quad \text{ for } |T-t|>>\frac{1}{\sigma}\\
\mathcal{G}f\simeq
\frac{\sigma^2}{2}f''-\sigma cf', \quad \text{ for } |T-t|<<\frac{1}{\sigma}.\\
\end{cases}
\]
We have two regimes depending on the distance between $t$ and the final observation time $T$. For a small $t$ and large $T$, the generator is close to the one of an Ornstein-Uhlenbeck process fluctuating around $0$ and for $t$ close to $T$, the generator is close to the one of the  drifted Brownian motion
$Y$ driving the population to the neighborhood of $-c$, as observed in the simulations.\\

At this point, we can give the law of the history of an uniformly sampled individual
when the initial condition was $F$. From the explicit value of the generator given in \eqref{eq:generatorG-explicite}, one can deduce that there exists a Brownian motion $( B_{t})_{t}$ independent of $Y_0$ such that for $0\leq t \leq T$, the process $Y$ satisfies the stochastic differential equation:
\be
\label{SDE-Y}
dY_{t}=-\sigma \tanh(\sigma(T-t))Y_{t}dt - \frac{\sigma c}{\cosh (\sigma(T-t))} dt+\sigma d B_{t}.
\ee

\begin{prop}\label{prop:Ytilde-resolu}The Markov process $Y$ with generator \eqref{eq:generatorG-explicite} is a Gaussian process which can be expressed explicitly for $0\leq t\leq T$:
\begin{multline}\label{eq:Ytilde-resolu}
Y_t= \frac{\cosh(\sigma(T-t))}{\cosh(\sigma T)}Y_0 + c \cosh(\sigma(T-t)) \big(\tanh(\sigma(T-t))-\tanh(\sigma T)\big)\\
+\sigma \cosh(\sigma(T-t)) \int_{0}^{t}\frac{dB_s}{\cosh(\sigma(T-s))} .
\end{multline}
\end{prop}

\begin{proof}
Because \eqref{SDE-Y} looks like an Ornstein-Uhlenbeck process, we define for all $t\le T$,
 \begin{equation}\label{eq:etape2}
 Z_{t}=e^{\int_{0}^{t}\sigma \tanh(\sigma(T-s))\ ds}Y_{t}.
 \end{equation}
Applying Itô's formula to \eqref{eq:etape2}, one obtains that:
 \begin{equation}\label{eq:etape4}
{Z}_{t}=  {Y}_{0} -\int_{0}^{t} \frac{\sigma c}{\cosh(\sigma(T-s))} e^{\int_{0}^{s}\sigma \tanh(\sigma(T-u))\ du} ds+\int_{0}^{t}\sigma e^{\int_{0}^{s}\sigma \tanh(\sigma(T-u))\ du} d B_{s}.
 \end{equation}

Since the primitive of $\tanh(x)$ is $\log \cosh (x)$, we obtain that:
\begin{equation}\label{eq:etape1}\int_{0}^{t}\sigma \tanh(\sigma(T-s))\ ds = \big[-\log \cosh(\sigma(T-s))\big]_0^t = \log \Big(\frac{\cosh (\sigma T)}{\cosh (\sigma(T-t))}\Big).
\end{equation}
Thus, it is possible to rewrite \eqref{eq:etape2} as
\begin{equation}\label{eq:etape3}
{Z}_{t} = \frac{\cosh (\sigma T)}{\cosh (\sigma(T-t))} {Y}_t,
\end{equation}and we obtain moreover from \eqref{eq:etape4} that
\begin{align}
 {Z}_{t}=  &  {Y}_0 - \int_0^t \frac{\sigma c \cosh(\sigma T)}{\cosh^2(\sigma(T-s))}ds +\int_{0}^{t}\frac{\sigma  \cosh(\sigma T)}{\cosh(\sigma(T-s))} dB_{s} \nonumber\\
  = &  {Y}_0 + c \cosh(\sigma T) \Big[\tanh(\sigma(T-s))\Big]_0^t + \sigma \cosh(\sigma T)\int_{0}^{t}\frac{dB_s}{\cosh(\sigma(T-s))} , \label{eq:etape5}
 \end{align}by using that a primitive of $1/\cosh^2(x)$ is $\tanh(x)$. Equations \eqref{eq:etape3} and \eqref{eq:etape5} give the announced result.
\end{proof}

\subsection{Return to the initial population process}

The spinal process $Y$ obtained in Theorems \ref{spine} and with generator given in \eqref{eq:generatorG-explicite} is associated to the auxiliary branching-diffusion process $\widetilde Z^K$. We have now to prove that it is close to its analogous for the initial population process $Z^K$. By Corollary \ref{proximity}, we know that when we start from the stationary measure, these two processes are uniformly (in time) close when $K$ is large, at least on a finite time interval. From this fact, we can obtain a similar result for $Z^K$ as the one enounced for $\widetilde{Z}^K$ in Corollary \ref{cor:hist-spine}.

 \begin{prop}\label{prop:hist}
When $(Z^K_0)$ satisfies \eqref{hyp:moment} and converges weakly and in probability to the stationary measure $F$ defined in \eqref{Fgauss}, then for any $\Phi$ continuous and bounded,
\begin{eqnarray}
\lim_{K\rightarrow +\infty} \E_{ Z^K_0}\left[\frac{\langle H^K_T,\Phi\rangle}{\langle H^K_T,1\rangle}\right]
&=&\lim_{K\to +\infty} \mathbb{E}\left[\frac{1}{\widetilde{N}^K_{T}}\sum_{i\in \widetilde{V}^K_{T}}\Phi(\widetilde X^i_{s},\ s\leq T) \ \bigg|\ \sum_{i\in \widetilde V^K_{0}} \delta_{\widetilde{X}^i_{0}}= K Z^K_0\right] \nonumber \\
&=& \int_{\mathbb{R}} m_{T}(x) \mathbb{E}_{x}\left[\Phi(Y_{s},\ s\leq T)\right] {F(dx)\over \lambda}.\label{eq:hist}\end{eqnarray}
\end{prop}

\begin{proof}{The right expression is obtained from Corollary 4.3 with $\xi_{0}=F$ and by using \eqref{mF}}. Let us consider $\varepsilon>0$. It is possible to find a cylindrical test-function $\varphi$ of the form \eqref{test-function} such that:
\begin{equation}\label{approx:function-test-cyl}
\sup_{y\in \Co(\R,\R)} \big|\Phi(y)-\varphi(y)\big|\leq \varepsilon.
\end{equation}
By \eqref{eq:spine-processuslinearise}, the right hand side of \eqref{eq:hist} is the limit when $K$ tends to infinity of
\begin{equation}
\label{etape3}
\mathbb{E}\left[\frac{1}{\widetilde{N}^K_{T}}\sum_{u\in \widetilde{V}^K_{T}}\Phi(\widetilde X^u_{s},\ s\leq T) \ \bigg|\ \sum_{i\in \widetilde V^K_{0}} \delta_{\widetilde{X}^i_{0}}= K Z^K_0 \right]
=\mathbb{E}_{Z^K_0}\left[\frac{\langle \widetilde{H}^K_T,\Phi\rangle}{\langle \widetilde{H}^K_T,1\rangle} \right].\end{equation}
To prove the proposition, it is hence sufficient to prove that the left hand side of \eqref{eq:hist} and \eqref{etape3} have the same limit. For this, we write:
\begin{multline}
\Big| \frac{\langle H^K_T,\Phi\rangle}{\langle H^K_T,1\rangle} - \frac{\langle \widetilde{H}^K_T,\Phi\rangle}{\langle \widetilde{H}^K_T,1\rangle}\Big| \leq  \Big| \frac{\langle H^K_T,\Phi\rangle}{\langle H^K_T,1\rangle} - \frac{\langle H^K_T,\varphi\rangle}{\langle H^K_T,1\rangle}\Big|
+  \Big|\frac{\langle H^K_T,\varphi\rangle}{\langle H^K_T,1\rangle}-\frac{\langle \widetilde{H}^K_T,\varphi\rangle}{\langle \widetilde{H}^K_T,1\rangle}\Big|\\
+\Big|\frac{\langle \widetilde{H}^K_T,\varphi\rangle}{\langle \widetilde{H}^K_T,1\rangle}-\frac{\langle \widetilde{H}^K_T,\Phi\rangle}{\langle \widetilde{H}^K_T,1\rangle}\Big|.\label{etape4}
\end{multline}
Notice that each of the fraction is upper-bounded by $\|\Phi\|_\infty$ or $\|\varphi\|_\infty$ (with the convention $0/0=0$) so that each of the terms in the right hand side is bounded. For the first term on the right hand side of \eqref{etape4}, we have by \eqref{approx:function-test-cyl}:
\begin{align*}
\E_{Z^K_0}\left[\Big| \frac{\langle H^K_T,\Phi\rangle}{\langle H^K_T,1\rangle} - \frac{\langle H^K_T,\varphi\rangle}{\langle H^K_T,1\rangle}\Big|\right]\leq &
\E_{Z^K_0}\left[\frac{1}{\langle H^K_T,1\rangle} \big| \langle H^K_T,\Phi-\varphi\rangle\big|\right]\nonumber\\
\leq & \|\Phi-\varphi\|_\infty \leq \varepsilon.
\end{align*}
Proceeding similarly, we can show that the third term is also upper bounded by $\varepsilon$. For the second term, let us first introduce the following stopping times, for ${1>}\eta>0$:
\[\tau^K_\eta=\inf\big\{t\in \R_+,\ \langle H^K_t,1\rangle \notin (\eta,1/\eta)\big\},\qquad
\widetilde{\tau}^K_\eta=\inf\big\{t\in \R_+,\ \langle \widetilde{H}^K_t,1\rangle \notin (\eta,1/\eta)\big\}.\]
Because the processes $(\langle H^K_t,1\rangle)_{t\in \R_+}=(\langle Z^K_t,1\rangle)_{t\in \R_+}$ and $(\langle \widetilde{H}^K_t,1\rangle)_{t\in \R_+}$ converge {to $\lambda$ (see Proposition \ref{prop:coupligHHtilde},  Corollary \ref{initialisation-F} and Equation \eqref{mF})}, we have that for $\eta$ small enough:
\[\lim_{K\rightarrow +\infty}\P(\tau^K_\eta\leq T)=\lim_{K\rightarrow +\infty}\P(\widetilde{\tau}^K_\eta\leq T)=0.\]
Thus, it is possible to choose $\eta$ such that both probabilities are smaller than $\varepsilon$. Then,
\begin{align*}
\Big| \frac{\langle H^K_T,\varphi\rangle}{\langle H^K_T,1\rangle}-\frac{\langle \widetilde{H}^K_T,\varphi\rangle}{\langle \widetilde{H}^K_T,1\rangle}\Big| \leq &
\big|\langle \widetilde{H}^K_T,\varphi\rangle \big| \times \Big|\frac{1}{\langle \widetilde{H}^K_t,1\rangle }- \frac{\ind_{\tau^K_\eta>T}\ind_{\widetilde{\tau}^K_\eta>T}}{\langle \xi_T,1\rangle }\Big|\\
+ &  \Big|\frac{\ind_{\tau^K_\eta>T}\ind_{\widetilde{\tau}^K_\eta>T}}{\langle \xi_T,1\rangle }\Big| \times \big|\langle \widetilde{H}^K_T,\varphi\rangle - \langle H^K_T,\varphi\rangle \Big|\\
+ & \big|\langle H^K_T,\varphi\rangle \big| \times \Big|\frac{1}{\langle {H}^K_t,1\rangle }- \frac{\ind_{\tau^K_\eta>T}\ind_{\widetilde{\tau}^K_\eta>T}}{\langle \xi_T,1\rangle }\Big|.
\end{align*}
For the first term in the right hand side, we use that:
\begin{equation*}
\big|\langle \widetilde{H}^K_T,\varphi\rangle \big| \times \Big|\frac{1}{\langle \widetilde{H}^K_t,1\rangle }- \frac{\ind_{\tau^K_\eta>T}\ind_{\widetilde{\tau}^K_\eta>T}}{\langle \xi_T,1\rangle }\Big|\leq  \ind_{\tau^K_\eta \wedge \widetilde{\tau}^K_\eta\leq T}
  \|\varphi\|_\infty + \frac{1}{\eta^2 \langle \xi_T,1\rangle}\Big|\langle H^K_T,1\rangle - \langle \xi_T,1\rangle\Big|, \end{equation*}
and taking the expectation:
\begin{align*}
\E_{Z^K_0}\Big[\big|\langle \widetilde{H}^K_T,\varphi\rangle \big| \times \Big|\frac{1}{\langle \widetilde{H}^K_t,1\rangle }- \frac{\ind_{\tau^K_\eta>T}\ind_{\widetilde{\tau}^K_\eta>T}}{\langle \xi_T,1\rangle }\Big|\Big]\leq 2 \|\varphi\|_\infty \varepsilon + \frac{1}{\eta^2 \langle \xi_T,1\rangle}\E\Big[\big|\langle \widetilde{H}^K_T,1\rangle - \langle \xi_T,1\rangle\big|\Big].
\end{align*}A similar upper-bound can be obtained for the third term. Gathering the latter bounds:
\begin{multline*}
\E_{Z^K_0}\Big[\Big| \frac{\langle H^K_T,\varphi\rangle}{\langle H^K_T,1\rangle}-\frac{\langle \widetilde{H}^K_T,\varphi\rangle}{\langle \widetilde{H}^K_T,1\rangle}\Big| \Big]\leq  4 \|\varphi\|_\infty \varepsilon + \frac{1}{\eta^2 \langle \xi_T,1\rangle}\E\Big[\big|\langle \widetilde{H}^K_T,1\rangle - \langle \xi_T,1\rangle\big|+\big|\langle H^K_T,1\rangle - \langle \xi_T,1\rangle\big|\Big]\\
+\frac{1}{\langle \xi_T,1\rangle}
\E\Big[\big|\langle \widetilde{H}^K_T,\varphi\rangle - \langle H^K_T,\varphi\rangle\big|\Big].
\end{multline*}
We can now conclude with Corollary \ref{initialisation-F} and Proposition \ref{prop:coupligHHtilde}.
\end{proof}

\begin{figure}[!ht]
\begin{center}
\begin{tabular}{ccc}
$t= 2 \Delta t$ & \includegraphics[width=5cm,height=3cm]{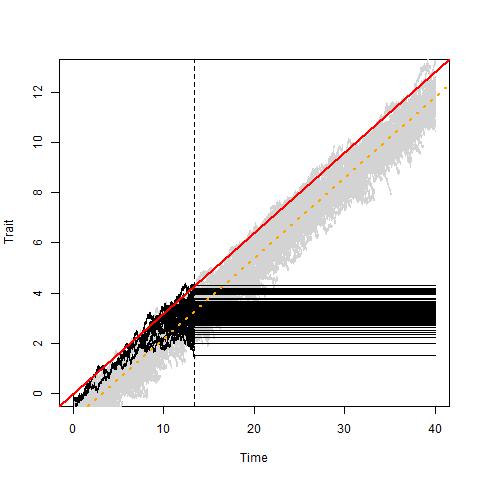} &
 \includegraphics[width=5cm,height=3cm]{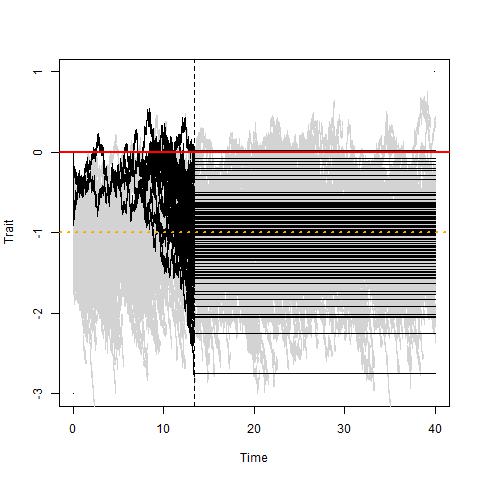} \\
$t= 3 \Delta t$ & \includegraphics[width=5cm,height=3cm]{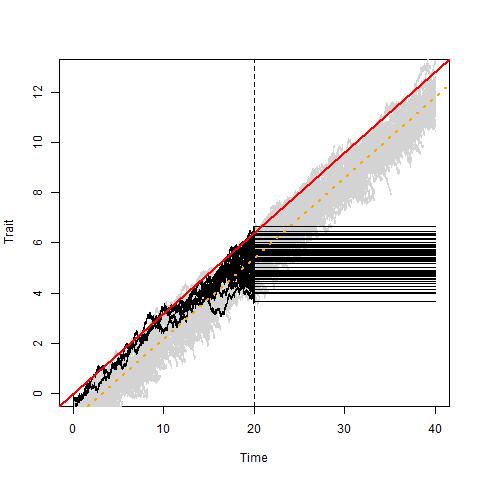} &
 \includegraphics[width=5cm,height=3cm]{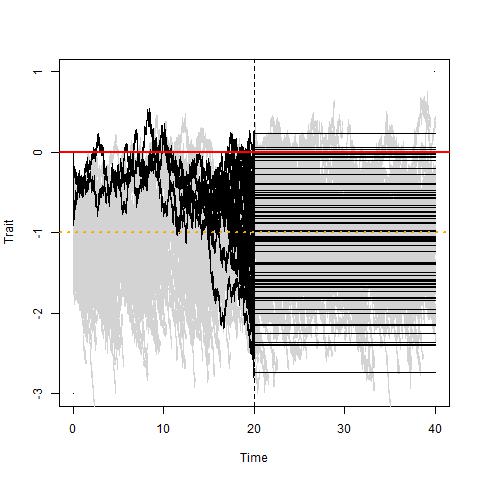} \\
$t= 4 \Delta t$ & \includegraphics[width=5cm,height=3cm]{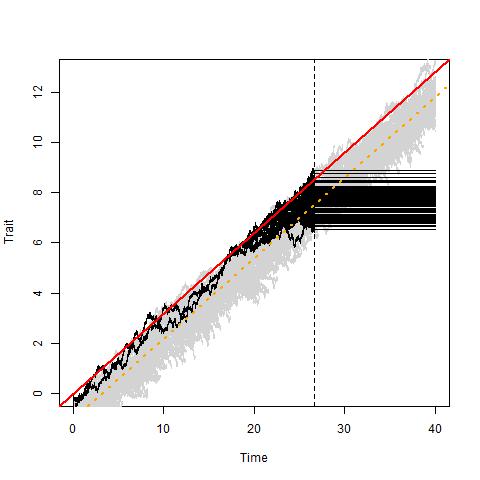} &
 \includegraphics[width=5cm,height=3cm]{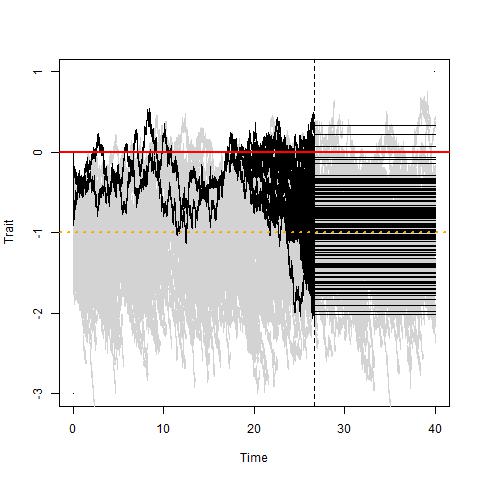} \\
$t= 5 \Delta t$ & \includegraphics[width=5cm,height=3cm]{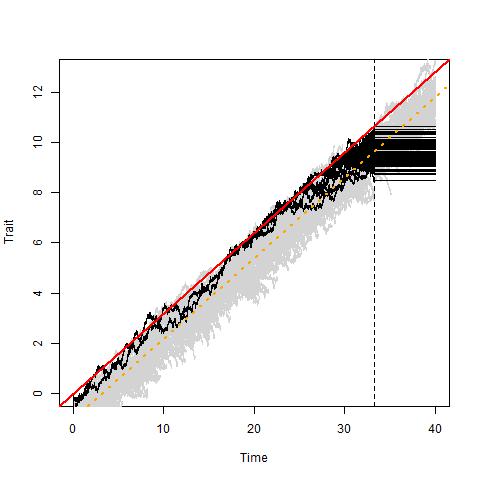} &
 \includegraphics[width=5cm,height=3cm]{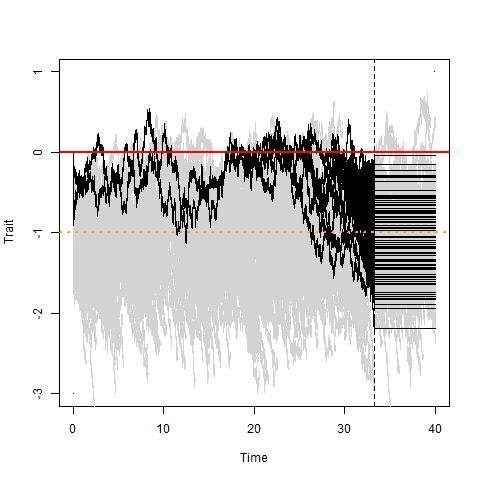} \\
$t= 6 \Delta t$ & \includegraphics[width=5cm,height=3cm]{Simu04122020-6.jpeg} &
 \includegraphics[width=5cm,height=3cm]{Simu04122020-mobile-6.jpeg} \\
& (a) & (b)
\end{tabular}
\caption{{\small Evolution of the ancestral lineages of the present population, for various times $t=k\Delta t$ with $\Delta t=20/3$ and $k\in \{2,\dots 6\}$. The traits in the population (ordinate) are shown with respect to time (abscissa). The extinct lineages are in gray, whereas the lineages of the living particles are in block. (a): fixed framework. (b): mobile framework.}}
\end{center}
\end{figure}

\subsection{The spinal time reversed equation}

Our purpose in this section is to recover the trait ancestor of an individual sampled in $F$ at time $T$, that is, in the population at time $T$ when the initial condition is the stationary solution $F$. For this, we need to reverse the time in the equation of the spinal process, and we will use to this purpose a result by Haussmann and Pardoux \cite{haussmannpardoux}. Their formula to reverse the diffusion \eqref{eq:Ytilde-resolu} requires the computation of the density of $Y_t$ for every time $t\geq 0$.

\me First, notice that:
 \begin{prop}The approximating {(for $K\rightarrow +\infty$)} distribution at time $0$ of a trait chosen uniformly in the population at time $T$ according to the stationary measure $F$ comes from a biased initial condition, $\frac{1}{\lambda}m_T(x)F(x)$ and not $\frac{1}{\lambda}F(x)$. Conditionally to $Y_T\leadsto F$,
 \be
 \label{Y0}
Y_{0}\sim \mathcal{N}\left(- c e^{-\sigma T},\ \frac{\sigma}{1+\tanh(\sigma T)}\right).
\ee
 \end{prop}

\begin{proof}
Applying  \eqref{eq:hist}  with $\Phi(x_{s},\ s\leq T) = f(x_{0})$, we obtain that the random variable $Y_{0}$ has the distribution
 $\frac{1}{\lambda}m_T(x)F(x)dx$. Computing this measure yields that $Y_{0}$ has a Gaussian law whose expectation and variance are respectively $-ce^{-\sigma T}$ and $\sigma/(1+\tanh(\sigma T))$.
 \end{proof}

\me Remember that $F$ is a Gaussian distribution centered in $-c$. The biaised distribution $\frac{1}{\lambda} m_{T}F$ describes the traits at time $0$  of the individuals producing  individuals alive at $T$. When $T$ is large, its support is in the tail of the distribution $F$.

\bigskip

\me We are now able to compute  the density of  $Y_{t}$, using Proposition \ref{prop:Ytilde-resolu}.\\

\begin{prop}
For any $0\leq t\leq T$, the random variable $Y_{t}$ is a normal variable with law
\[
Y_{t}\sim \mathcal{N}\left(-ce^{-\sigma (T-t)},\ \frac{\sigma}{1+\tanh(\sigma(T-t))}\right)
\]
whose density $p(t,x)$ is given by
\[
\partial_{x}\log p(t,x)=-\frac{x+ce^{-\sigma(T-t)}}{\sigma}\left(1+\tanh(\sigma(T-t)) \right).
\]
\end{prop}

\me \begin{proof}
From \eqref{eq:Ytilde-resolu} and \eqref{Y0}, we deduce that $Y_{t}$ has a normal law.

\begin{align*}
\mathbb{E}[Y_{t}]&=-ce^{-\s T}\frac{\cosh(\s(T-t))}{\cosh(\s T)}+c\cosh(\s(T-t))\left\{\tanh(\s (T-t))-\tanh(\s T) \right\}\\
&=c\cosh(\s (T-t))\left(-\frac{e^{-\s T}}{\cosh(\s T)}-\tanh(\s T)+\tanh(\s(T-t)) \right).
\end{align*}
Since
\begin{equation}
\label{eq:1}
\tanh(\s T)-1=\frac{e^{\s T}-e^{-\s T}-e^{\s T}-e^{-\s T}-}{e^{\s T}+e^{-\s T}}=-\frac{e^{-\sigma T}}{\cosh(\s T)},
\end{equation}
we have
\begin{align*}
\mathbb{E}[ Y_{t}]&=c\cosh(\s (T-t))\left(\tanh(\s(T-t))-1 \right)\\
&=-c\cosh(\s (T-t))\left(\frac{e^{-\s(T-t)}}{\cosh(\sigma (T-t))} \right)=-ce^{-\s(T-t)}.
\end{align*}

Additionally,
\begin{align*}
\text{var}({Y}_{t})&=\cosh(\s (T-t))^{2}\left({\frac{\s}{\cosh^{2}(\s T)\left(1+\tanh(\s T) \right)}}+{\s^{2}\int_{0}^{t}\frac{1}{\cosh^{2}(\s (T-s))}\ ds} \right)\\
&= \cosh(\s (T-t))^{2}\left(\frac{\s e^{-\s T}}{\cosh(\s T)}+ \sigma\left(\tanh(\sigma T)-\tanh(\s (T-t)) \right)\right)\\
&= \cosh(\s (T-t))^{2}\, \s(1-\tanh(\s(T-t)))\\
&=\sigma \cosh(\s(T-t))e^{-\s(T-t)}=\frac{\s}{1+\tanh(\s(T-t))}.
\end{align*}

The result follows.\end{proof}

We are now able to obtain the time reversed equation giving the trajectory leading from the trait of a ``typical'' individual living at time $T$ in the stationary distribution $F$, to its ancestor.

\begin{prop}
\label{reversed:2}
The time  reversed process of the spinal process $ Y$ is the  time homogeneous Ornstein-Uhlenbeck process driving the ancestral trajectories around $0$, satisfying the  equation
\begin{equation}\label{rev-spine:OU}
d\widehat{Y}_{s}=-\sigma \widehat{Y}_{s}ds +\sigma d W_{s},
\end{equation}
for a Brownian motion $W$.
\end{prop}

\begin{proof}
To reverse time in the equation \eqref{SDE-Y}, we apply  an explicit formula given in \cite{haussmannpardoux}. The reverse process will be a diffusion process with the same diffusion coefficient $\sigma
$ and with a new drift term
\[
b^{r}(t,x)=-b(T-t,x)+\sigma^{2}\partial_{x}\log p(T-t,x),
\]

where $p(t,.)$ is the density of $ Y_{t}$ and $b(t,x)$ is the drift term in \eqref{SDE-Y}:
\[
b(t,x)=-\sigma\tanh(\s(T-t))x-\frac{\s c}{\cosh\left(\s(T-t)\right)}.
\]

We obtain

\begin{align*}
b^{r}(t,x)= & \sigma\tanh(\s t)x+\frac{\s c}{\cosh\left(\s t \right)} {-\sigma x\left(\tanh(\s t) \right)-\sigma x-\sigma ce^{-\s t}\left(1+\tanh(\s t) \right)}\\
= & -\sigma x,
\end{align*}
by using \eqref{eq:1}. The reverse process is then a very simple time homogeneous Ornstein-Uhlenbeck process driving the ancestral trajectories around $0$, satisfying Equation \eqref{rev-spine:OU} for a Brownian motion $W$.
\end{proof}

As a consequence of Propositions \ref{prop:hist} and \ref{reversed:2}, we can now summarize our results in the following.

\begin{thm}
	\label{thm:conclusion}
Let $U_{K}$ be a random variable whose conditional distribution with respect to $H^{K}_{T}$ is uniform on $V^{K}_{T}$ and consider the processes $(\widehat{Y}^{K}_{s})_{0\leq s\leq T}$ defined by
\[
\widehat{Y}^{K}_{s}=X^{U_{K}}_{T-s}, \quad \forall s\in[0,T]^.
\]
Then, under the hypotheses of Proposition \ref{prop:hist}, the processes $\widehat{Y}^{K}$ converges, as $K$ goes to infinity, weakly to $\widehat{Y}$ in $\Co([0,T],\mathbb{R})$.
\end{thm}

{\footnotesize
\providecommand{\noopsort}[1]{}\providecommand{\noopsort}[1]{}\providecommand{\noopsort}[1]{}\providecommand{\noopsort}[1]{}

}

\appendix

\section{SDEs for the stochastic birth-death particle system and the historical particle system}
\label{app:sde-section}

\subsection{Pathwise representation of the population process}
\label{app:sde}
We recall here the pathwise representation of our mesure-valued processes, as solution of stochastic differential equations driven by inependent Poisson point measures and Brownian motions. We refer to
\cite{champagnatmeleard} and \cite{meleardtran_suphist} for more details.

To model the random occurrence of birth and death events, let us consider a Poisson point process $N(ds,di,d\theta)$ on $\R_+\times \mathcal{I}\times \R_+$, with intensity measure $ds\otimes n(di)\otimes d\theta$, where $n(di)$ is the counting measure on the set of labels $\mathcal{I}=\bigcup_{n\in\mathbb{N}}\N^n$.\\
We also introduce a family of independent standard Brownian motions $\big(B^i, i\in \mathcal{I}\big)$ indexed by $\mathcal{I}$ that will drive the particle motions.\\

The atoms of the Poisson point process determine birth and death events. These events modify the set of individuals alive, $V^K_t$. Between these events, the position of a particle alive, say $i$, is modelled by a drifted diffusion
\begin{equation}
dX^i_t=\sigma(dB^i_t-c\ dt).\label{eq:Xi}
\end{equation}
Let us consider a test function $f \in \Co_b^{1,2}(\R_+\times \R)$. We will use the notation $f_t(x)=f(t,x)$. Between two jump times, the set of living individuals is fixed and we can apply Itô's formula to the diffusion processes \eqref{eq:Xi} related to the individuals $i$ alive. At a jump time $\tau$, if we have a birth of individual $i$, a new offspring appears at the same position and the process increases from $f_\tau(X^i_\tau)$. If we have a death of individual $i$, the process decreases of $f_\tau(X^i_\tau)$. Then the measure-valued population process $Z^K$ acts on the test function $f$ as:

\begin{align}
\lefteqn{\langle Z^{K}_t,f_t\rangle =  \langle Z^{K}_0,f_0\rangle}\nonumber\\
  + & \frac{1}{K}\int_0^t \int_{\mathcal{I}}\int_{\R_+}\ind_{\{i\in V^K_{s-}\}} \Big[ f_{s}(X^i_{s})\ind_{\{\theta\leq 1\}} - f_{s}(X^i_{s}) \ind_{\{1<\theta\leq 1+\frac{(X^i_s)^2}{2}+\langle Z_{s-}^{K},1\rangle\}} \Big] N(ds, di, d\theta)\nonumber \\
+ & \frac{1}{K}\sum_{i\in \mathcal{I}}\int_0^t  \ind_{\{i\in V^K_s\}} \sigma \partial_x f_s(X^i_s) dB^i_s \nonumber\\
+ & \frac{1}{K} \sum_{i\in \mathcal{I}}\int_0^t \ind_{\{i\in V^K_s\}} \Big( \partial_s f_s(X^i_s) -c\sigma \partial_x f_s(X^i_s)+\frac{\sigma^2}{2} \partial^2_{xx}f_s(X^i_s) \Big) ds,\label{eq:sdeZ}
\end{align}
and where the set of living individuals is changing as follows.
\begin{itemize}
\item $V^K_0=\{1,\ldots,K\}$ and $|V^K_0|=K\langle Z^K_0,1\rangle$.
\item For each atom $(s,i,\theta)$ of $N$ such that $i\in V^K_{s_-}$ and $\theta\leq 1$, there is a new birth by individual $i$, and the label of the new offspring is $j=(i,k)$ where $k$ is the rank of the new individual among the daughters of $i$.
\item For each atom $(s,i,\theta)$ of $N$ such that $i\in V^K_{s_-}$ and $1<\theta\leq 1+(X^i_s)^2/2+|V^K_{s-}|/K$, there is a death and the label $i$ is removed from $V^K_{s_-}$.
\end{itemize}

Introducing the compensated martingale measure of the Poisson point measure, we obtain that
\begin{multline}
\langle Z^{K}_{t},f_t\rangle=\langle Z^{K}_0,f_0\rangle +M^{K,\varphi}_{t}
+\int_{0}^{t} \int_{\mathbb{R}}  \left\{ \left(1-\frac{1}{2}x^2-\langle Z^{K}_{s}, 1\rangle\right)f_s(x)\right.\\
\left.+\partial_sf_s(x)-\sigma c \partial_xf_s(x)+\frac{\sigma^2}{2}\partial^2_{xx}f_s(x)\right\} Z^{K}_{s}(dx)\, ds, \label{ZK} \end{multline}
where the process $\,M^{K,\varphi}\,$ is a square integrable martingale with quadratic variation process given by
\begin{equation}
\label{crochet}
\langle M^{K,\varphi} \rangle_{t} = {1\over K} \int_{0}^t \int_{\R} \Big\{  \Big(1+{x^2\over 2} + \langle Z^{c,K}_{s}, 1\rangle\Big)f_s^2(x) + {\sigma^2} (\partial_x f_s)^2(x) \Big\} Z^{K}_{s}(dx) ds.
\end{equation}

\subsection{Pathwise representation of the historical  population process}
\label{app:sde-hist}

\medskip Let us consider test functions $\varphi$ defined on $\mathbb{R}_{+}\times \cal{C}(\R_+,\R)$ with a similar form as in  \eqref{test-function}, i.e. for any $s,y\in \mathbb{R}_{+}\times \cal{C}(\R_+,\R)$,
$$\varphi(s,y)=\varphi_s(y) =  \prod_{j=1}^m g_{j}(s, y_{t_{j}}) ,$$
for $m\in \mathbb{N}^*$, $0\le t_{1}<\cdots <t_{m}$ and $\forall j \in \{1,\cdots,m\}$, $g_{j}\in C_{b}^{1,2}(\mathbb{R}_{+}\times\mathbb{R}, \mathbb{R})$. Note that
$$\varphi(s,y_{.\wedge s})=\prod_{j=1}^m g_{j}(s, y_{t_{j}\wedge s}) =  \sum_{k=0}^{m-1}\ind_{[t_k,t_{k+1})}(s)\Big(\prod_{j=1}^k g_j(s, y_{t_j}) \, \prod_{j=k+1}^m g_j(s, y_s)\Big).$$
It is possible to write a stochastic differential equation for the historical process $H^K$ defined in \eqref{def:HKt} that is driven by the same Poisson point measures and Brownian motion as the process $Z^K$. With the notation \eqref{derivee} and \eqref{laplacien} introduced in Section \ref{section:coupling_historique}, we have

\begin{align}
&\langle H^{K}_t,\varphi_t\rangle =  \langle H^{K}, \varphi_0\rangle+ \frac{1}{K}\sum_{i\in \mathcal{I}}\int_0^t  \ind_{\{i\in V^K_s\}} \sigma {\widetilde{D}}\varphi_s(X^i_s) dB^i_s \nonumber\\
  + & \frac{1}{K}\int_0^t \int_{\mathcal{I}}\int_{\R_+}\ind_{\{i\in V^K_{s-}\}} \Big[ \varphi_{s}(X^i_{(.\wedge s)})\ind_{\{\theta\leq 1\}} -\varphi_{s}(X^i_{(.\wedge s)}) \ind_{\{1<\theta\leq 1+\frac{(X^i_{s})^2}{2}+\langle H_{s-}^{K},1\rangle\}} \Big] N(ds, di, d\theta)\nonumber \\
+ & \frac{1}{K} \int_0^t \sum_{i\in V^K_s}\Big( \partial_s \varphi_s(X^i_{.\wedge s}) -c\sigma {\widetilde{D}} \varphi_s(X^i_{(.\wedge s)})+\frac{\sigma^2}{2} {\widetilde{\Delta}}\varphi_s(X^i_{(.\wedge s)}) \Big) ds.\label{eq:sdeprev}
\end{align}

Then introducing the compensated martingales measures associated with the Poisson point processes, we obtain that

\begin{align}
& \langle H^{K}_t,\varphi_t\rangle =  \langle H^{K}, \varphi_0,\rangle+  \frac{1}{K} \int_0^t  \sum_{i\in V^K_{s}} \sigma {\widetilde{D}} \varphi_s(X^i_{(.\wedge s)}) dB^i_s \nonumber\\
  + & \frac{1}{K}\int_0^t \int_{\mathcal{I}}\int_{\R_+}\ind_{\{i\in V^K_{s-}\}} \Big[ \varphi_{s}(X^i_{(.\wedge s)})\ind_{\{\theta\leq 1\}} -\varphi_{s}(X^i_{(.\wedge s)}) \ind_{\{1<\theta\leq 1+\frac{(X^i_s)^2}{2}+\langle H_{s-}^{K},1\rangle\}} \Big] \tilde N(ds, di, d\theta)\nonumber \\
+ & \frac{1}{K}\int_0^t \sum_{i\in V^K_{s}} \Big(1 - \frac{(X^i_s)^2}{2}+\langle H_{s-}^{K},1\rangle \Big) \varphi_{s}(X^i_{(.\wedge s)}) ds\nonumber \\
+ & \frac{1}{K} \int_0^t \sum_{i\in V^K_s}\Big( \partial_s \varphi_s(X^i_{(.\wedge s)}) -c\sigma {\widetilde{D}}\varphi_s(X^i_{(.\wedge s)})+\frac{\sigma^2}{2} {\widetilde{\Delta}}\varphi_s(X^i_{(.\wedge s)}) \Big) ds\nonumber\\
& \hskip 2cm =  \langle H^{K}, \varphi_0,\rangle+ \mathcal{ M}^K_t(\varphi) \nonumber\\
& +
 \int_0^t \int_{\Co(\R_+,\R)} \Big( \big( 1- \frac{y_s^2}{2} - \langle H^K_s,1\rangle\big)\varphi(s,y)  +  \partial_s \varphi_s(y) +
\frac{\sigma^2}{2} \widetilde{\Delta}\varphi(s,y)-\sigma c \widetilde{D}\varphi(s,y) \Big)H^K_s(dy)\, ds.\label{eq:sde}
\end{align}

The process is a square integrable local martingale with quadratic variation
\begin{align}
&\langle \mathcal{ M}^K(\varphi) \rangle_{t} =  \frac{1}{K}  \int_0^t  \int_{\Co(\R_+,\R)} \Big( \big( 1+ \frac{y_s^2}{2} + \langle H^K_s,1\rangle\big)\varphi(s,y)  +  \sigma^2 ({\widetilde{D}}\varphi(s,y) )^2
 \Big)H^K_s(dy)\, ds.\label{eq:mart}
\end{align}

\subsection{Stochastic mild equation}
\label{app:sde-mild}

\label{app:stoch-mild} Recall that $(P_t)_{t\geq 0}$ is the semi-group defined in \eqref{eq:semi-gp}. For a fixed $t >0$ and a test function $\varphi\in \Co_b^2(\R)$, choosing $f(s,x)=P_{t-s}\varphi(x)$, we obtain from \eqref{eq:sdeZ} a mild stochastic equation:
\begin{align}
\lefteqn{\langle Z^{K}_t,\varphi\rangle =  \langle Z^{K}_0,P_t\varphi\rangle}\nonumber\\
  + & \frac{1}{K}\int_0^t \int_{\mathcal{I}}\int_{\R_+}\ind_{\{i\in V^K_{s_-}\}} \Big[ P_{t-s}\varphi(X^i_s)\ind_{\{\theta\leq 1\}} - P_{t-s}\varphi(X^i_s) \ind_{\{1<\theta\leq 1+\frac{(X^i_s)^2}{2}+\langle Z_{s_-}^{K},1\rangle\}} \Big] N(ds, di, d\theta)\nonumber \\
+ & \frac{1}{K}\sum_{i\in \mathcal{I}}\int_0^t  \ind_{\{i\in V^K_s\}} \sigma \partial_x P_{t-s}\varphi(X^i_s) dB^i_s\nonumber\\
= & \langle Z^{K}_0,P_t\varphi\rangle+ \int_0^t \int_{\R} \big(1-\frac{x^2}{2}-\langle Z^{K},1\rangle\big) P_{t-s}\varphi(x) Z^{K}_s(dx)\ ds + \mathcal{M}_t^{K,\varphi}\label{eq:sde-mild}
\end{align}where $\mathcal{M}_t^{K,\varphi}$ is the following square integrable martingale:
\begin{align}
\mathcal{M}_t^{K,\varphi}= & \frac{1}{K}\int_0^t \int_{\mathcal{I}}\int_{\R_+}\ind_{\{i\in V^K_{s_-}\}} \Big[ P_{t-s}\varphi(X^i_s)\ind_{\{\theta\leq 1\}} \nonumber\\
 & - P_{t-s}\varphi(X^i_s) \ind_{\{1<\theta\leq 1+\frac{(X^i_s)^2}{2}+\langle Z_{s_-}^{K},1\rangle\}} \Big] \big(N(ds, di, d\theta)-ds\otimes n(di)\otimes d\theta \big)\nonumber\\
+ & \frac{1}{K}\sum_{i\in \mathcal{I}}\int_0^t  \ind_{\{i\in V^K_s\}} \sigma \partial_x P_{t-s}\varphi(X^i_s) dB^i_s.\label{app:martingale-mild}
\end{align}The predictable quadratic variation of $\mathcal{M}_t^{K,\varphi}$ is
\begin{align}
\langle \mathcal{M}^{K,\varphi}\rangle_t=&  \frac{1}{K} \int_0^t \big(1+\frac{x^2}{2}+\langle Z_s^{K},1\rangle\big) (P_{t-s}\varphi(x))^2 \ Z^{K}_s(dx)\ ds\nonumber\\
+ & \frac{1}{K} \int_0^t \sigma^2 \big\langle Z^{K}_s, \big(\partial_x P_{t-s}\varphi \big)^2 \big\rangle ds.\label{app:crochet-mild}
\end{align}

\section{Moment estimates for $Z^{K}$: proof of Lemma \ref{lem:propagation-moment}}\label{app:moment}

We  prove a more precise form of Lemma \ref{lem:propagation-moment}.

 \begin{lem}\label{lem:propagation-moment2} We assume that the initial condition $Z^{K}_{0}$ satisfies for $\epsilon>0$ that:
\begin{equation}\label{hyp:momentB}
\sup_{K\in \N^*}\E\big(\langle Z^K_0,1\rangle^{2+\epsilon} \big)<+\infty\qquad \mbox{ and }\qquad \sup_{K\in \N^*} \E\big( \langle Z^K_0,x^2\rangle^{1+\epsilon}\big)<+\infty.
\end{equation}
Then,  for any $T>0$, we have
\begin{equation}\sup_{K\in \N^*} \E\big(\sup_{t\in [0,T]} \langle Z^K_t,1\rangle^{2+\epsilon}\big)<+\infty\qquad \mbox{ and }\qquad \sup_{K\in \N^*} \sup_{t\in [0,T]} \E\big( \langle Z^K_t,x^2\rangle^{1+\epsilon}\big)<+\infty.\label{eq:propagationB}\end{equation}
{Under the additional assumption that:
\begin{equation}\label{hyp:moment2B}
\sup_{K\in \N^*}\E\big(\langle Z^K_0,x^4\rangle^{1+2\epsilon}\big)<+\infty,\end{equation}
we also have that:
\begin{equation}
\sup_{K\in \N^*} \E\big( \sup_{t\in [0,T]}  \langle Z^K_t,x^2\rangle^{1+\epsilon}\big)<+\infty.\label{eq:propagation2B}
\end{equation}}
\end{lem}

 Using classical computation (see e.g. \cite{fourniermeleard}), several moment estimates can be derived under Assumption \eqref{hyp:momentB}. Recall that $T>0$ and assume \eqref{hyp:momentB}, i.e. that the initial condition $Z^{K}_{0}$ satisfies for $\epsilon>0$ that:
\begin{equation*}
\sup_{K\geq 1}\E\big(\langle Z^K_0,1\rangle^{2+\epsilon} \big)<+\infty\qquad \mbox{ and }\qquad \sup_{K\in \N^*} \E\big( \langle Z^K_0,x^2\rangle^{1+\epsilon}\big)<+\infty.
\end{equation*}

\noindent \textbf{Step 1:} Let us introduce the stopping time, for $M>0$ and for $K\geq 1$:
\begin{equation}
\tau^K_M = \inf\big\{t\geq 0,\ \langle Z^{K}_t,1\rangle^{2+\epsilon}>M \quad \mbox{ or }\quad \langle Z^K_t,x^2\rangle^{1+\epsilon} > M\big\}.
\end{equation}

Choosing the test function $\varphi \equiv 1$ and neglecting the natural death term of rate $x^2/2$ gives in \eqref{ZK}:
\begin{equation*}
\langle Z^K_{t\wedge \tau^K_M},1\rangle \leq   \langle Z^K_0,1\rangle + \int_0^{t\wedge \tau^K_M} \big(\langle Z^K_s,1\rangle - \langle Z^K_s,1\rangle^2\big) ds+ M^{K,1}_{t\wedge \tau^K_M}.
\end{equation*}Taking the expectation and using the convexity of $x\mapsto x^2$, it follows that
\begin{align*}
\E\big(\langle Z^K_t,1\rangle\big)\leq & \E\big(\langle Z^K_0,1\rangle\big) + \int_0^{t} \Big[\E\big(\langle Z^K_{s\wedge \tau^K_M},1\rangle\big) - \E\big(\langle Z^K_s,1\rangle\big)^2\Big] ds\\
\leq & \frac{\E\big(\langle Z^K_0,1\rangle\big)}{\E\big(\langle Z^K_0,1\rangle\big)+\big(1-\E\big(\langle Z^K_0,1\rangle\big)\big) e^{-t}},
\end{align*}since we recognize the logistic equation. Because the upper-bound does not depend on $M$, a direct consequence is that $\tau^K_M$ tends a.s. to infinity when $M\rightarrow +\infty$ and that:
\begin{equation}
 \sup_{t\in \R_+} \E\big(\langle Z^K_t,1\rangle\big)<+\infty.\label{etape6}
\end{equation}

\noindent \textbf{Step 2:} Now, choosing the test function $\varphi(x)=1$, using Itô's formula (see e.g. \cite[p.66]{ikedawatanabe}) and neglecting the death terms:
\begin{align*}
\langle Z^{K}_{t \wedge \tau^K_M},1\rangle^{2+\epsilon} \leq & \langle Z^{K}_0,1\rangle^{2+\epsilon} + \int_0^{t \wedge \tau^K_M} \int_{\mathcal{I}}\int_{\R}\ind_{i\in V^K_{s_-}} \Big(\big(\langle Z^K_{s-},1\rangle + \frac{1}{K}\big)^{2+\epsilon} -\langle Z^K_{s-},1\rangle^{2+\epsilon}\Big) \ind_{\theta\leq 1} N(ds,di,d\theta)\\
\leq & \langle Z^{K}_0,1\rangle^{2+\epsilon} + \int_0^{t \wedge \tau^K_M} \int_{\mathcal{I}}\int_{\R}\ind_{i\in V^K_{s_-}} \frac{C}{K}\langle Z^K_{s-},1\rangle^{1+\epsilon}  \ind_{\theta\leq 1} N(ds,di,d\theta),\label{etape7}
\end{align*}for a constant $C>0$ and using that
\begin{equation}\label{etape8}
\big(x+\frac{1}{K}\big)^{2+\epsilon}-x^{2+\epsilon}=x^{2+\epsilon}\Big[\exp\big((2+\epsilon)\ln\big(1+\frac{1}{Kx}\big)\big)-1\Big]\stackrel{K\rightarrow +\infty}{\sim} x^{2+\epsilon} \times \frac{2+\epsilon}{xK}=\frac{({2}+\epsilon) x^{1+\epsilon}}{K}.\end{equation}
Introducing the supremum in the right hand side, then in the left hand side and taking the expectation provides that:
\begin{align*}
\E\Big(\sup_{s\leq t}\langle Z^{K}_{s \wedge \tau^K_M},1\rangle^{2+\epsilon}\Big)\leq & \E\big(\langle Z^{K}_0,1\rangle^{2+\epsilon}\big) + C \int_0^{t } \E\Big(\sup_{u\leq s}\langle Z^K_{u},1\rangle^{2+\epsilon}\Big)  ds,
\end{align*}from which we obtain by Gronwall's lemma that:
\begin{align*}
\E\Big(\sup_{s\leq t}\langle Z^{K}_{s \wedge \tau^K_M},1\rangle^{2+\epsilon}\Big)\leq & \E\big(\langle Z^{K}_0,1\rangle^{2+\epsilon}\big)\exp\big(C t\big),
\end{align*}where the upper bound does not depend on $M$ nor on $K$. Then, letting $M\rightarrow +\infty$ provides the first estimate of \eqref{eq:propagationB}.\\

Notice that a similar computation would have yielded that:
\begin{equation}\sup_{K\in \N^*}\E\Big(\sup_{s\leq t}\langle Z^{K}_{s},1\rangle^{1+\epsilon}\Big)< +\infty. \label{etape12}
\end{equation}

\noindent \textbf{Step 3:} Let us now consider the test function $\varphi(x)=x^2$. Using Itô's formula and neglecting the death terms, we obtain from \eqref{ZK}:
\begin{align}
\lefteqn{\langle Z^K_{t\wedge \tau^K_M},x^2\rangle^{1+\epsilon} \leq   \langle Z^K_0,x^2\rangle^{1+\epsilon}} \nonumber\\
+ &  \int_0^{t\wedge \tau^K_M} \int_{\mathcal{I}} \int_{\R_+}  \ind_{\{i\in V^K_{s-}\}}  \ind_{\{\theta\leq 1\}} \Big(\big(\langle Z^K_{s-},x^2\rangle + \frac{(X^i_{s})^2}{K}\big)^{1+\epsilon}-\langle Z^K_{s-},x^2\rangle^{1+\epsilon}\Big) N(ds,di,d\theta)\nonumber\\
+ & \frac{1}{K}  \sum_{i\in \mathcal{I}}\int_0^{t\wedge \tau^K_M} \ind_{\{i\in V^K_s\}}2 \sigma(1+\epsilon)\langle Z^K_s,x^2\rangle^\epsilon  X^i_s \ dB^i_s
+  \frac{1}{{K^2}} \sum_{i\in \mathcal{I}} \int_0^{t\wedge \tau^K_M} \ind_{\{i\in V^K_s\}} \frac{\epsilon(1+\epsilon)}{\langle Z^K_s,x^2\rangle^{1-\epsilon } }2\sigma^2 (X^i_s)^2 \ ds \nonumber\\
+ & \frac{1}{K} \sum_{i\in \mathcal{I}} \int_0^{t\wedge \tau^K_M} \ind_{\{i\in V^K_s\}} (1+\epsilon)\langle Z^K_s,x^2\rangle^\epsilon \Big(\sigma^2-2c\sigma X^i_s\Big) ds \nonumber
\end{align}
\begin{align}
\lefteqn{\langle Z^K_{t\wedge \tau^K_M},x^2\rangle^{1+\epsilon} \leq  \langle Z^K_0,x^2\rangle^{1+\epsilon}+  \frac{1}{K}  \sum_{i\in \mathcal{I}}\int_0^{t\wedge \tau^K_M} \ind_{\{i\in V^K_s\}}2 \sigma(1+\epsilon)\langle Z^K_s,x^2\rangle^\epsilon  X^i_s \ dB^i_s }\nonumber\\
+ &  \int_0^{t\wedge \tau^K_M} \int_{\mathcal{I}} \int_{\R_+}  \ind_{\{i\in V^K_{s-}\}}  \ind_{\{\theta\leq 1\}} \langle Z^K_{s-},x^2\rangle^{1+\epsilon} \Big(\big(1 + \frac{(X^i_{s})^2}{K \langle Z^K_{s-},x^2\rangle}\big)^{1+\epsilon}-1\Big) N(ds,di,d\theta) \nonumber\\
+ & \epsilon(1+\epsilon){2\over K}\sigma^2 \int_0^t \langle Z^K_{s\wedge \tau^K_M},x^2\rangle^{{\epsilon}} ds
+  (1+\epsilon) \int_0^{t\wedge \tau^K_M} \langle Z^K_s,x^2\rangle^\epsilon \times \big(\sigma^2\langle Z^K_s,1\rangle - 2c\sigma \langle Z^K_s,x\rangle\big) ds\label{etape9}
\end{align}
First, because $x\leq 1+x^2$, we have that:
\begin{align*}
 \big|\langle Z^k_s,x^2\rangle^\epsilon \times \big(\sigma^2\langle Z^K_s,1\rangle - 2c\sigma \langle Z^K_s,x\rangle\big) \big| \leq  &
 2c\sigma \langle Z^K_s,x^2\rangle^{1+\epsilon} + (\sigma^2+2c \sigma )\langle Z^K_s,x^2\rangle^\epsilon \langle Z^K_s,1\rangle\\
 \leq & (\sigma^2+ 4c\sigma) \langle Z^K_s,x^2\rangle^{1+\epsilon} + (\sigma^2+2c \sigma ) \langle Z^K_s,1\rangle^{1+\epsilon}.
 \end{align*}
Then, notice that a computation similar to \eqref{etape8} gives that for a constant $C>0$ sufficiently large,
\[
\langle Z^K_{s-},x^2\rangle^{1+\epsilon} \Big(\big(1 + \frac{(X^i_{s})^2}{K \langle Z^K_{s-},x^2\rangle}\big)^{1+\epsilon}-1\Big)
\leq \langle Z^K_{s-},x^2\rangle^\epsilon \frac{C(1+\epsilon)}{K} (X^i_{s})^2.
\]
Gathering these results in \eqref{etape9}:
\begin{align}
\langle Z^K_{t\wedge \tau^K_M},x^2\rangle^{1+\epsilon}\leq & \langle Z^K_0,x^2\rangle^{1+\epsilon} +
2\sigma^2\epsilon(1+\epsilon)t + (1+\epsilon)\big(C+ \sigma^2+ 4c\sigma\big) \int_0^t \langle Z^K_{s\wedge \tau^K_M},x^2\rangle^{1+\epsilon}  ds\nonumber\\
+  & (\sigma^2+2c\sigma)T \sup_{s\leq T}\langle Z^K_s,1\rangle^{1+\epsilon}+M^K_{t\wedge\tau^K_M}\label{etape11}
\end{align}where $(M^K_{t\wedge \tau^K_M})_{t\geq 0}$ is a square integrable martingale. Taking the expectation, using Gronwall's lemma and \eqref{etape12} implies that:
\begin{align}
 \sup_{t\in [0,T]} \E\big(\langle Z^K_{t\wedge \tau^K_M},x^2\rangle^{1+\epsilon}\big)\leq &  \sup_{K\in \N^*} \Big(\E\big(\langle Z^K_0,x^2\rangle^{1+\epsilon}\big)+2\sigma^2\epsilon(1+\epsilon) T+(\sigma^2+2c\sigma)T \sup_{s\leq T}\langle Z^K_s,1\rangle^{1+\epsilon}\Big)\times \nonumber\\
  & \times \exp\big( T(1+\epsilon)\big(C + \sigma^2+ 4c\sigma\big)\big).\label{etape15}
\end{align}Because the right hand side does not depend on $M$ for $K$, we obtain:
\begin{equation}
\sup_{K\in \N^*}\sup_{t\in [0,T]} \E\big(\langle Z^K_{t},x^2\rangle^{1+\epsilon}\big)<+\infty.\label{etape16}
\end{equation}
A similar computation yields that under the additional assumption \eqref{hyp:moment2B}, we also have:
\begin{equation}
\sup_{K\in \N^*}\sup_{t\in [0,T]} \E\big(\langle Z^K_{t},x^4\rangle^{1+\epsilon}\big)<+\infty.\label{etape17}
\end{equation}

\noindent \textbf{Step 4:} Now, let us take the supremum in \eqref{etape11}:
\begin{align}
\sup_{s\leq t}\langle Z^K_{s\wedge \tau^K_M},x^2\rangle^{1+\epsilon}\leq & \langle Z^K_0,x^2\rangle^{1+\epsilon} +2\sigma^2\epsilon(1+\epsilon) t +
 (1+\epsilon)\big(C+  \sigma^2+ 4c\sigma\big) \int_0^t \sup_{u\leq s} \langle Z^K_{u\wedge \tau^K_M},x^2\rangle^{1+\epsilon}  ds\nonumber\\
+ &  (\sigma^2+2c\sigma)T \sup_{s\leq T}\langle Z^K_s,1\rangle^{1+\epsilon}+ \sup_{s\leq t} M^K_{s\wedge\tau^K_M}\label{etape14}
\end{align}
The bracket of the martingale $(M^K_{t\wedge \tau^K_M})_{t\geq 0}$ is:
\begin{align}
\langle M^K\rangle_{t\wedge \tau^K_M} = & \frac{4\sigma^2(1+\epsilon)^2}{K}  \int_0^{t\wedge \tau^K_M} \langle Z^K_s,x^2\rangle^{1+2\epsilon} ds \nonumber\\
+ & \int_0^{t\wedge \tau^K_M} \sum_{i\in V^K_s} \Big(\big(\langle Z^K_{s-},x^2\rangle + \frac{(X^i_{s})^2}{K}\big)^{1+\epsilon}-\langle Z^K_{s-},x^2\rangle^{1+\epsilon}\Big)^2 ds\nonumber\\
\leq & \int_0^{t\wedge \tau^K_M} \Big( \frac{4\sigma^2(1+\epsilon)^2}{K}  \langle Z^K_s,x^2\rangle^{1+2\epsilon}+ \frac{C^2(1+\epsilon)^2}{K}\langle Z^K_s,x^2\rangle^{2\epsilon}\langle Z^K_s,x^4\rangle \Big)ds \nonumber\\
\leq &  \int_0^{t\wedge \tau^K_M} \Big( \big(\frac{4\sigma^2(1+\epsilon)^2}{K}+ \frac{C^2(1+\epsilon)^2}{K}\big) \langle Z^K_s,x^2\rangle^{1+2\epsilon}+ \frac{C^2(1+\epsilon)^2}{K}\langle Z^K_s,x^4\rangle^{1+2\epsilon} \Big)ds .
\end{align}
Thus, using Doob's lemma:
\begin{align}
\E\big(\sup_{s\leq t} M^K_s\big)\leq  & 4 \E\Big(\int_0^{t\wedge \tau^K_M} \Big( \big(\frac{4\sigma^2(1+\epsilon)^2}{K}+ \frac{C^2(1+\epsilon)^2}{K}\big) \langle Z^K_s,x^2\rangle^{1+2\epsilon}+ \frac{C^2(1+\epsilon)^2}{K}\langle Z^K_s,x^4\rangle^{1+2\epsilon} \Big)ds \Big)\nonumber\\
\leq & \frac{C(T)}{K},\label{etape18}
\end{align}by \eqref{etape16} and \eqref{etape17}. Now, taking the expectation in \eqref{etape14}, and using Gronwall's inequality with \eqref{etape12} and \eqref{etape18} yields that
\begin{equation}
\sup_{K\in \N^*} \E\big(\sup_{t\in [0,T]}\langle Z^K_{t},x^2\rangle^{1+\epsilon}\big)<+\infty.
\end{equation}

\section{Proof of Theorem \ref{spine}}

\begin{proof}
We denote here $h(x)=1-\frac{x^2}{2}-\lambda$. Notice that the proof here holds for any function $h$ that is upper bounded (but not necessarily lower bounded).\\

For $x\in \R$, $T>0$ and $t\leq T$, let us define the following measure for a test function $\Phi$ continuous and bounded on $C([0,T],\mathbb{R})$, where $Y$ is the diffusion process defined in \eqref{def:Y}:
\begin{equation}
\langle \mu_t^{T,x},\Phi\rangle=\frac{\E_x\Big(\exp\big(\int_0^T h(X_s)ds\big)\Phi(X_s,s\leq t)\Big)}{\E_x\Big(\exp\big(\int_0^T h(X_s)ds\big)\Big)}.
\end{equation}
Let us prove that under $\mu_t^{T,x}$, the canonical process is an inhomogeneous Markov process with infinitesimal generator \eqref{eq:generatorG}.\\

Denoting $\mathbb{E}^{\mu^{T,x}}$ the expectation under $\mu^{T,x}$, we have that, for some real numbers $t$ and $u$ s.t.\ $t\geq u\geq 0$,
	\begin{equation}
	\label{eq:condEx}
	\mathbb{E}^{\mu^{T,x}}\left[f(X_{t})\bigg| \mathcal{F}_{u} \right]=\frac{\mathbb{E}\left[f(X_{t})\exp\left(\int_{0}^{T} h(X_{s})ds \right)\bigg| \mathcal{F}_{u} \right]}{\mathbb{E}\left[\exp\left(\int_{0}^{T} h(X_{s})ds \right)\bigg| \mathcal{F}_{u}  \right]}.
	\end{equation}

Markov property for $X$ under $\mathbb{P}_{x}$ and Formula \eqref{tenan} entail that\begin{equation}
\label{trick}\mathbb{E}\left[\exp\left(\int_{0}^{T}h(X_{s})\ ds \right)\Bigg| \mathcal{F}_{u}\right] = m_{T-u}(X_{u}) \exp\left(\int_{0}^{u}h(X_{s})\ ds\right).
\end{equation}
We also have

\begin{align}
&\mathbb{E}\left[f(X_{t})\exp\left(\int_{0}^{T}h(X_{s})ds \right)\Bigg| \mathcal{F}_{u} \right]\nonumber\\
&= \mathbb{E}\left[f(X_{t})\exp\left(\int_{0}^{T}h(X_{s})ds \right)\Bigg| \mathcal{F}_{t} \Bigg| \mathcal{F}_{u} \right]\nonumber\\
&= \exp\left(\int_{0}^{u}h(X_{s})\ ds\right)\mathbb{E}\left[f(X_{t})\exp\left(\int_{u}^{t}h(X_{s})\ ds \right)m_{T-t}(X_{t})\Bigg| \mathcal{F}_{u} \right].\label{eq:markov}
\end{align}

\medskip

\noindent
Now, as $f$ and $m$ are smooth (cf. Corollary \ref{avril}) and $X$ is a Markov process with generator $L$ (see Lemma \ref{lem:MTO}), we have for any $0\le u\le t$,
\begin{equation}
\label{eq:thmeq2}
f(X_{t})m_{T-t}(X_{t})=f(X_{u})m_{T-u}(X_{u})+\int_{u}^{t} \big( L(fm_{T-s})(X_{s}) -f(X_{s})\partial_{t}m_{T-s}(X_{s})\big)\ ds+M_{t},
\end{equation}
where $M$ is some $\mathbb{P}$-martingale started at $0$. Thus, applying It\^o's  formula, we get
\begin{alignat*}{2}
f(X_{t})&\, m_{T-t}(X_{t})\exp\left(\int_{u}^{t}h(X_{s})ds \right) =f(X_{u})m_{T-u}(X_{u})\\&+\int_{u}^{t}\bigg(f(X_{s})\partial_{t}m_{T-s}(X_{s})+L(fm_{T-s})(X_{s})-f(X_{s})m_{T-s}(X_{s})h(X_{s})\bigg) \exp\left(\int_{u}^{s}h(X_{v})\ dv \right)\ ds\\
&+\int_{u}^{t}\exp\left(\int_{u}^{s}h(X_{v})\ dv \right) dM_{s}.
\end{alignat*}
Using \eqref{m-edp}  gives
\begin{alignat*}{2}
f(X_{t})&m_{T-t}(X_{t})\exp\left(\int_{u}^{t}h(X_{s})ds \right) =f(X_{u})m_{T-u}(X_{u})\\&+\int_{u}^{t}\bigg(L(fm_{T-s})(X_{s})-f(X_{s})Lm_{T-s}(X_{s})\bigg) \exp\left(\int_{u}^{s}h(X_{v})\ dv \right)\ ds\\
&+\int_{u}^{t}\exp\left(\int_{u}^{s}h(X_{v})\ dv \right) dM_{s}.
\end{alignat*}
Using  \eqref{trick} and notation  \eqref{eq:generatorG}, we finally obtain
\begin{align*}
\lefteqn{\mathbb{E}\left[f(X_{t})m_{T-t}(X_{t})\exp\left(\int_{0}^{t}h(X_{s})ds \right)\Bigg| \mathcal{F}_{u}\right] }\\
=&f(X_{u})m_{T-u}(X_{u})\exp\left(\int_{0}^{u}h(X_{s})\ ds \right)\\&+\mathbb{E}\left[\int_{u}^{t}\mathcal{G}_{s}f(X_{s})m_{T-s}(X_{s}) \exp\left(\int_{0}^{s}h(X_{v})\ dv \right)\ ds\Bigg| \mathcal{F}_{u}\right]\\
=&f(X_{u})\mathbb{E}\left[\exp\left(\int_{0}^{T}h(X_{s})\ ds \right)\Bigg| \mathcal{F}_{u}  \right] \\&+\mathbb{E}\left[\int_{u}^{t}\mathcal{G}_{s}f(X_{s}) \mathbb{E}\left[\exp\left(\int_{0}^{T}h(X_{v})\ dv \right)\Bigg| \mathcal{F}_{s}\right]\ ds\Bigg| \mathcal{F}_{u}\right].
\end{align*}
Thus, using \eqref{eq:condEx} and \eqref{eq:markov}, we have
\begin{equation*}
\mathbb{E}^{\mu^{T,x}}\left[f(X_{t})\Bigg| \mathcal{F}_{u}\right] =f(X_{u})+\mathbb{E}^{\mu^{T,x}}\left[\int_{u}^{t}\mathcal{G}_{s}f(X_{s}) \Bigg| \mathcal{F}_{u}\right].
\end{equation*}
This ends the proof.
\end{proof}

\end{document}